\documentclass[11pt,a4paper]{article}
\usepackage[utf8]{inputenc}

\usepackage{graphicx}
\usepackage{url}

\title{Stacked polytopes and tight triangulations of manifolds}
\date{\today}
\author{Felix~Effenberger\footnote{Institut f\"ur Geometrie und Topologie, Universit\"at Stuttgart, 70550 Stuttgart, Germany, \texttt{effenberger@mathematik.uni-stuttgart.de}.}}

\newcommand{\squishlist}{
 \begin{list}{$\bullet$}
  { \setlength{\itemsep}{0pt}
     \setlength{\parsep}{3pt}
     \setlength{\topsep}{3pt}
     \setlength{\partopsep}{0pt}
     \setlength{\leftmargin}{1.5em}
     \setlength{\labelwidth}{1em}
     \setlength{\labelsep}{0.5em} } }

\newcommand{\squishlisttwo}{
 \begin{list}{$\bullet$}
  { \setlength{\itemsep}{0pt}
     \setlength{\parsep}{0pt}
    \setlength{\topsep}{0pt}
    \setlength{\partopsep}{0pt}
    \setlength{\leftmargin}{2em}
    \setlength{\labelwidth}{1.5em}
    \setlength{\labelsep}{0.5em} } }

\newcommand{\squishend}{
  \end{list}  }

\usepackage{enumerate}

\usepackage{amsmath}
\usepackage{amsfonts}
\usepackage{amssymb}  
\usepackage{amstext}
\usepackage{wasysym}
\usepackage{upref}

\usepackage[amsmath,thmmarks,thref,hyperref]{ntheorem}

\numberwithin{equation}{section}

\theoremstyle{plain}
\theoremnumbering{arabic}

\theoremheaderfont{\normalfont\bfseries}
\newtheorem{definition}{Definition}
\newtheorem{lemma}[definition]{Lemma}

\newtheorem{theorem}[definition]{Theorem}
\newtheorem{corollary}[definition]{Corollary}
\newtheorem{conjecture}[definition]{Conjecture}

\theorembodyfont{\upshape}
\newtheorem{question}[definition]{Question}
\newtheorem{remark}[definition]{Remark}

\theoremstyle{nonumberplain}
\theoremheaderfont{\normalfont\itshape}
\theorembodyfont{\normalfont}
\theoremseparator{.}
\theoremsymbol{\ensuremath{\square}}
\newtheorem{proof}{Proof}

\numberwithin{definition}{section}

\newcommand{\iso}{\ensuremath{\cong}}
\renewcommand{\epsilon}{\varepsilon}

\newcommand{\ra}{\ensuremath{~\rightarrow~}}

\newcommand{\set}[1]{\ensuremath{\{#1\}}}

\newcommand{\N}[1][]{\mathbb{N}^{#1}}

\newcommand{\R}[1][]{\mathbb{R}^{#1}}
\newcommand{\C}[1][]{\mathbb{C}^{#1}}

\DeclareMathOperator*{\skel}{skel}

\newcommand{\lk}{\operatorname{lk}}
\newcommand{\str}{\operatorname{st}}

\newcommand{\Span}{\operatorname{span}}

\newcommand*{\dtimes}{\mathrel{\vcenter{\offinterlineskip\hbox{$\times$}\vskip-.55ex\hbox{$\hspace*{.3ex}\underline{\hspace*{1.2ex}}$}}}}

\begin{document}
\maketitle
\vspace*{-6cm}
\begin{flushright}
\it
to appear in Journal of Combinatorial Theory, Series A
\end{flushright}
\vspace*{5cm}

\begin{abstract}
	Tightness of a triangulated manifold is a topological condition, roughly meaning that any simplex-wise linear embedding of the triangulation into Euclidean space is ``as convex as possible''. It can thus be understood as a generalization of the concept of convexity. In even dimensions, su\-per-\-neigh\-bor\-li\-ness is known to be a purely combinatorial condition which implies the tightness of a triangulation.
		Here, we present other sufficient and purely combinatorial conditions which can be applied to the odd-dimensional case as well. One of the conditions is that all vertex links are stacked spheres, which implies that the triangulation is in Walkup's class $\mathcal{K}(d)$.
		We show that in any dimension $d\geq 4$, \emph{tight-neighborly} triangulations as defined by Lutz, Sulanke and Swartz 	are tight.
		Furthermore, triangulations with $k$-stacked vertex links and the centrally symmetric case are discussed.
	
	
\end{abstract}

\section{Introduction and results}
\label{sec:intro}

	Tightness is a notion developed in the field of differential geometry as the equality of the (normalized) \emph{total absolute curvature} of a submanifold with the lower bound \emph{sum of the Betti numbers} \cite{Kuiper84GeomTotAbsCurvTheo, Banchoff97TightSubmSmoothPoly}. It was first studied by Alexandrov \cite{Alexandrov38ClassClosedSurf}, Milnor \cite{Milnor50RelBettiHypersurfIntGaussCurv}, Chern and Lashof \cite{Chern57TotCurvImmMnf} and Kuiper \cite{Kuiper59ImmMinTotAbsCurv} and later extended to the polyhedral case by Banchoff \cite{Banchoff65TightEmb3DimPolyMnf}, Kuiper \cite{Kuiper84GeomTotAbsCurvTheo} and Kühnel \cite{Kuehnel95TightPolySubm}.
	
	From a geometrical point of view, tightness can be understood as a generalization of the concept of convexity that applies to objects other than topological balls and their boundary manifolds since it roughly means that an embedding of a submanifold is ``as convex as possible'' according to its topology. The usual definition is the following.

	\begin{definition}[tightness \cite{Kuiper84GeomTotAbsCurvTheo, Kuehnel95TightPolySubm}]
		Let $\mathbb{F}$ be a field. An embedding $M \rightarrow \mathbb{E}^N$ of a compact manifold is called \emph{$k$-tight with respect to $\mathbb{F}$} if for any open or closed half-space  $h\subset E^N$ the induced homomorphism
		
		$$H_i(M\cap h;\mathbb{F})\longrightarrow H_i(M;\mathbb{F})$$
		
		\noindent is injective for all $i\leq k$. $M$ is called \emph{$\mathbb{F}$-tight} if it is $k$-tight for all $k$. The standard choice for the field of coefficients is $\mathbb{F}_2$ and an $\mathbb{F}_2$-tight embedding is called \emph{tight}.
		\label{def:tighthom}
	\end{definition}

	With regard to PL embeddings of PL manifolds, the 	tightness of \emph{combinatorial manifolds} can also be defined via a purely combinatorial condition as follows. For an introduction to PL topology see 	\cite{Rourke72IntrPLTop}, for more recent developments in the field see \cite{Lutz05TrigMnfFewVertCombMnf, Datta07MinTrigManifolds}.
	
	\begin{definition}[combinatorial manifold, combinatorial tightness \cite{Kuehnel95TightPolySubm}]\hfill
		\begin{enumerate}[(i)]
			\item A simplicial complex $K$ that has a topological manifold as its underlying set $|K|$ is called \emph{triangulated manifold}. $K$ is called \emph{combinatorial manifold} of dimension $d$ if all vertex links of $K$ are PL $(d-1)$-spheres, where a PL $(d-1)$-sphere is a triangulation of the $(d-1)$-sphere that carries a standard PL structure.
			
			\item Let $\mathbb{F}$ be a field. A combinatorial manifold $K$ on $n$ vertices is called \emph{($k$-)tight w.r.t. $\mathbb{F}$} if its canonical embedding $$K\subset \Delta^{n-1}\subset E^{n-1}$$ is ($k$-)tight w.r.t. $\mathbb{F}$, where $\Delta^{n-1}$ denotes the $(n-1)$-dimensional simplex.
		\end{enumerate}
		\label{def:tighthomcomb}
	\end{definition}
	
	In dimension $d=2$ the following are equivalent for a triangulated surface $S$ on $n$ vertices: (i) $S$ has a complete edge graph $K_n$, (ii) $S$ appears as a so called \emph{regular case} in Heawood's Map Color Theorem \cite{Heawood90MapColThm, Ringel74MapColThm}, compare \cite[Chap.~2C]{Kuehnel95TightPolySubm} and (iii) the induced piecewise linear embedding of $S$ into Euclidean $(n-1)$-space has the two-piece property \cite{Banchoff74TightPolyKleinBottProjPlMoeb}, and it is tight  \cite{Kuehnel80Tight0TightPolyhEmbSurf}, \cite[Chap.~2D]{Kuehnel95TightPolySubm}.

	K\"uhnel investigated the tightness of combinatorial triangulations of manifolds also in higher dimensions and codimensions, see \cite{Kuehnel94ManSkelConvPolyt}, \cite[Chap.~4]{Kuehnel95TightPolySubm}. It turned out that the tightness of a combinatorial triangulation is closely related to the concept of \emph{Hamiltonicity} of a polyhedral complexes (see \cite{Kuehnel91HamSurfPoly, Kuehnel95TightPolySubm}): A subcomplex $A$ of a polyhedral complex $K$ is called \emph{$k$-Hamiltonian}\footnote{This is not to be confused with the notion of a $k$-Hamiltonian graph, see  \cite{Chartrand69CubeConnGraph1Ham}.} if $A$ contains the full $k$-dimensional skeleton of $K$. This generalization of the notion of a Hamiltonian circuit in a graph seems to be due to Schulz \cite{Schulz74DissMgfZell, Schulz94PolyhMnfPoly}. A Hamiltonian circuit then becomes a special case of a $0$-Hamiltonian subcomplex of a $1$-dimensional graph or of a higher-dimensional complex \cite{Ewald73HamCircSimpComp}.
	
	A triangulated $2k$-manifold that is a $k$-Hamiltonian subcomplex of the boundary complex of some higher dimensional simplex is a tight triangulation as Kühnel \cite[Chap.~4]{Kuehnel95TightPolySubm} showed. Such a triangulation is also called \emph{$(k+1)$-neighborly triangulation} since any $k+1$ vertices in a $k$-dimensional simplex are common neighbors. Moreover, $(k+1)$-neighborly triangulations of $2k$-manifolds are also referred to as \emph{super-neighborly} triangulations --- in analogy with neighborly polytopes the boundary complex of a $(2k+1)$-polytope can be at most $k$-neighborly unless it is a simplex. Notice here that combinatorial $2k$-manifolds can go beyond $k$-neighborliness, depending on their topology. 
	
	With the simplex as ambient polytope there exist generalized Heawood inequalities in even dimensions $d\geq 4$ that were first conjectured by Kühnel \cite{Kuehnel94ManSkelConvPolyt, Kuehnel95TightPolySubm}, almost completely proved in \cite{Novik98UBTHomMnf} by Novik and proved by Novik and Swartz in \cite{Novik08SocBuchsMod}. As in the $2$-dimensional case, the $k$-Hamiltonian triangulations of $2k$-manifolds here appear as regular cases of the generalized Heawood inequalities.
	
	There also exist generalized Heawood inequalities for $k$-Hamiltonian subcomplexes of cross polytopes that were first conjectured by Sparla \cite{Sparla99LBTComb2kMnf} and almost completely proved by Novik in \cite{Novik05OnFNumMnfSymm}. The subcomplexes appearing as regular cases in these inequalities admit a tight embedding into a higher dimensional cross polytope and are also referred to as \emph{nearly $(k+1)$-neighborly} as they contain all $i$-simplices, $i\leq k$, not containing one of the diagonals of the cross polytope (i.e. they are ``neighborly except for the diagonals of the cross polytope'').
	
	For $d=2$, a regular case of Heawood's inequality corresponds to a  triangulation of an abstract surface (cf. \cite{Ringel74MapColThm}). Ringel \cite{Ringel55WieManGeschlNichtoFl} and Jungerman and Ringel \cite{Jungerman80MinTrigOrientSurf} showed that all of the infinitely many regular cases of Heawood's inequality distinct from the Klein bottle do occur. As any such case yields a tight triangulation (see \cite{Kuehnel80Tight0TightPolyhEmbSurf}), there are infinitely many tight triangulations of surfaces. 
	
	In contrast, in dimensions $d\geq 3$ there only exist a finite number of known examples of tight triangulations (see \cite{Kuehnel99CensusTight} for a census), apart from the trivial case of the boundary of a simplex and an infinite series of triangulations of sphere bundles over the circle due to K\"uhnel  \cite[5B]{Kuehnel95TightPolySubm}, \cite{Kuehnel86HigherDimCsaszar}.

	Especially in odd dimensions it seems to be hard to give combinatorial conditions for the tightness of a triangulation and such conditions were not known so far. This work presents one such condition holding in any dimension $d\geq 4$.
	
	\bigskip
	
	\noindent In the course of proving the Lower Bound Conjecture (LBC) for $3$- and $4$-manifolds, D.~Walkup \cite{Walkup70LBC34Mnf} defined a class  $\mathcal{K}(d)$ of ``certain especially simple'' \cite[p.~1]{Walkup70LBC34Mnf} combinatorial manifolds as the set of all combinatorial $d$-manifolds that only have \emph{stacked} $(d-1)$-spheres as vertex links as defined below.

	\begin{definition}[stacked polytope, stacked sphere \cite{Walkup70LBC34Mnf}]\hfill
		\begin{enumerate}[(i)] 
									\item A simplex is a \emph{stacked} polytope and each polytope obtained from a stacked polytope by adding a pyramid over one of its facets is again stacked.
			
			\item A triangulation of the $d$-sphere $S^d$ is called \emph{stacked $d$-sphere} if it is combinatorially isomorphic to the boundary complex of a stacked $(d+1)$-polytope. 
			
					\end{enumerate}
	\end{definition}	
	
	Thus, a stacked $d$-sphere can be understood as the combinatorial manifold obtained from the boundary of the $(d+1)$-simplex by successive stellar subdivisions of facets of the boundary complex $\partial \Delta^{d+1}$ of the $(d+1)$-simplex (i.e. by successively subdividing facets of a complex $K_i$, $i=0,1,2,\dots$, by inner vertices, where  $K_0=\partial \Delta^{d+1}$).
	In this work we will give combinatorial conditions for the tightness of  members of $\mathcal{K}(d)$ holding in all dimensions $d\geq 4$. The main results of this paper are the following:
	
	In Theorem~\ref{thm:MorseTight} we show that any polar Morse function subject to a condition on the number of critical points of even and odd indices is a perfect function. This can be understood as a combinatorial analogon to Morse's lacunary principle, see Remark~\ref{rem:lacunaryprinciple}.
	
	This result is used in Theorem~\ref{thm:Kd2NeighborlyTight} in which it is shown that every $2$-neighborly member of $\mathcal{K}(d)$ is a tight triangulation for $d\geq 4$. Thus, all \emph{tight-neighborly} triangulations as defined in \cite{Lutz08FVec3Mnf} are tight for $d\geq 4$ (see Section~\ref{sec:TightNeigh}). 
	
	\noindent The paper is organized as follows.
	
	Section~\ref{sec:PolarMorseTight} begins with a short introduction to polyhedral Morse theory giving rise to a tightness definition of a triangulation in terms of (polyhedral) Morse theory, followed by the investigation on a certain family of perfect Morse functions. The latter functions can be used to give a combinatorial condition for the tightness of odd-dimensional combinatorial manifolds in terms of properties of the vertex links of such manifolds.
	
	In Section~\ref{sec:TightnessKd}, the tightness of members of $\mathcal{K}(d)$ is discussed, followed by a discussion of the tightness of tight-neighborly triangulations for $d\geq 4$ in Section~\ref{sec:TightNeigh}. Both sections include examples of triangulations for which the stated theorems hold.
	
	In Section~\ref{sec:Kkd}, the classes $\mathcal{K}^k(d)$ of combinatorial manifolds are introduced as a generalization of Walkup's class $\mathcal{K}(d)$ and examples of manifolds in these classes are presented. 	Furthermore, an analogue of Walkup's theorem \cite[Thm.~5]{Walkup70LBC34Mnf},  \cite[Prop.~7.2]{Kuehnel95TightPolySubm} for $d=6$ is proved, assuming the validity of the Generalized Lower Bound Conjecture \ref{conj:GLBC}.  Finally, Section~\ref{sec:SubcomplCross} focuses on subcomplexes of cross polytopes that lie in the class $\mathcal{K}^k(d)$ for some $k$. Here, an example of a centrally symmetric triangulation of $S^4\times S^2\in \mathcal{K}^2(6)$ as a $2$-Hamiltonian subcomplex of the $8$-dimensional cross polytope is given. This triangulation is part of a conjectured series of triangulations of sphere products as tight subcomplexes of cross polytopes.

\section{Polar Morse functions and tightness}
\label{sec:PolarMorseTight}
		
	Apart from the homological definition given in Definition~\ref{def:tighthom} and \ref{def:tighthomcomb}, tightness can also be defined in the language of Morse theory in a natural way:	On one hand, the total absolute curvature of a smooth immersion $X$ equals the average number of critical points of any non-degenerate height function on $X$ in a suitable normalization. On the other hand, the Morse inequality shows that the normalized total absolute curvature of a compact smooth manifold $M$ is bounded below by the rank of the total homology $H_{*}(M)$ with respect to any field of coefficients, where tightness is equivalent to the case of equality in this bound, see \cite{Kuehnel99CensusTight}.
	
	As an extension to classical Morse theory (see \cite{Milnor63MorseTheory} for an introduction to the field), K\"uhnel \cite{Kuehnel90TrigMnfFewVert, Kuehnel95TightPolySubm} developed what one might refer to as a ``polyhedral Morse theory''. Note that in this theory many, but not all concepts carry over from the smooth to the polyhedral case, see the survey articles \cite{Kuiper84GeomTotAbsCurvTheo} and \cite{Banchoff97TightSubmSmoothPoly} for a comparison of the two cases.
 
	A discrete analogon to the Morse functions in classical Morse theory, are defined in the polyhedral case as follows.
	
	\begin{definition}[rsl functions, \cite{Kuehnel90TrigMnfFewVert, Kuehnel95TightPolySubm}]
		Let $M$ be a combinatorial manifold of dimension $d$. A function $f: M\ra \R$ is called \emph{regular simplex-wise linear} (\emph{rsl}, for short), if $f(v)\neq f(v')$ for any two vertices $v\neq v'$ of $M$ and $f$ is linear when restricted to any simplex of $M$. Regular simplex-wise linear functions are sometimes also referred to as \emph{Morse functions}. 
	\end{definition}
	
	Notice that an rsl function is uniquely determined by its value on the set of vertices and that only vertices can be critical points of $f$ in the sense of Morse theory. With this definition at hand one can define critical points and levelsets of these Morse functions as in classical Morse theory. 
	
	\begin{definition}[critical vertices, \cite{Kuehnel90TrigMnfFewVert, Kuehnel95TightPolySubm}]
		Let $\mathbb{F}$ be a field, $M$ be a combinatorial $d$-manifold and let $f$ be an rsl function on $M$. A vertex $v\in M$ is called \emph{critical of index $k$ and multiplicity $m$ with respect to $f$}, if
				\begin{equation*}
			\dim_{ \mathbb{F}} H_k(M_v,M_v\backslash\set{v};\mathbb{F})=m>0,
		\end{equation*}
				\noindent where $M_v:=\set{x\in M:f(x)\leq f(v)}$ and $H_{*}$ denotes an appropriate homology theory with coefficients in $\mathbb{F}$. The \emph{number of critical points of $f$ of index $i$} (with multiplicity) are
		
		\begin{equation*}
			\mu_i(f; \mathbb{F}):=\sum_{v\in V(M)} \dim_{\mathbb{F}} H_i(M_v,M_v\backslash\set{v};\mathbb{F}).
		\end{equation*}
	\end{definition}
	
	In the following we will be interested in special kinds of Morse functions, so called \emph{polar} Morse functions. This term was coined by Morse, see \cite{Morse60ExistPolNonDegenFuncDiffMnf}. 

	\begin{definition}[polar Morse function]
		Let $f$ be a Morse function that only has one critical point of index $0$ and of index $d$ each for a given (necessarily connected) $d$-manifold. Then $f$ is called \emph{polar Morse function}. 
	\end{definition}
	
	Note that for a $2$-neighborly combinatorial manifold clearly all rsl functions are polar. As in the classical theory, there hold Morse relations as follows.
	
	\begin{theorem}[Morse relations, \cite{Kuehnel90TrigMnfFewVert, Kuehnel95TightPolySubm}]
		Let $\mathbb{F}$ be a field, $M$ a combinatorial manifold of dimension $d$ and $f$ an rsl function on $M$. Then the following holds, where $\beta_i(M; \mathbb{F}):=\dim_{ \mathbb{F}} H_i(M; \mathbb{F})$ denotes the $i$-th Betti number:
		
		\begin{enumerate}[(i)]
			\item $\mu_i(f; \mathbb{F})\geq \beta_i(M; \mathbb{F})$ for all $i$,
			\item $\sum_{i=0}^{d} (-1)^{i} \mu_i(f; \mathbb{F})=\chi(M)=\sum_{i=0}^{d} (-1)^{i} \beta_i(M; \mathbb{F})$,
			\item $M$ is ($k$-)tight with respect to $\mathbb{F}$ if and only if $\mu_i(f; \mathbb{F})=\beta_i(M; \mathbb{F})$ for every rsl function $f$ and for all $0\leq i \leq d$ (for all $0\leq i\leq k$).
		\end{enumerate}

		Functions satisfying equality in (i) for all $i \leq k$ are called \emph{$k$-tight functions}. A function $f$ that satisfies equality in (i) for all $i$ is usually referred  to as \emph{perfect} or \emph{tight} function, cf. \cite{Bott80MorseTheoYangMills}.   
		\label{thm:MorseRelations}
	\end{theorem}
	
	Note that a submanifold $M$ of $E^d$ is tight in the sense of Definition~\ref{def:tighthom} if and only if every Morse function on $M$ is a tight function, see \cite{Kuehnel90TrigMnfFewVert, Kuehnel95TightPolySubm}.
	
	As already mentioned in Section~\ref{sec:intro}, there exist quite a few examples of triangulations in even dimensions that are known to be tight, whereas ``for odd-dimensional manifolds it seems to be difficult to transform the tightness of a polyhedral embedding into a simple combinatorial condition'', as K\"uhnel \cite[Chap.~5]{Kuehnel95TightPolySubm} observed. Consequently, there are few examples of triangulations of odd-dimensional manifolds that are known to be tight apart from the sporadic triangulations in \cite{Kuehnel99CensusTight} and Kühnel's infinite series of $S^{d-1}\dtimes S^1$ for odd $d\geq 3$.
	
	It is a well known fact, that in even dimensions a Morse function which only has critical points of even indices is a tight function, cf. \cite{Bott80MorseTheoYangMills}. This follows directly from the Morse relations, i.e. the fact that $\sum_i (-1)^i\mu_i=\chi(M)$ holds for any Morse function on a manifold $M$ and the fact that $\mu_i\geq \beta_i$. In odd dimensions on the other hand, argumenting in this way is impossible as we always have $\mu_0\geq 1$ and the alternating sum allows the critical points to cancel out each other. 	
	What will be shown in Theorem~\ref{thm:MorseTight} is that at least for a certain family of Morse functions the tightness of its members can readily be determined in arbitrary dimensions $d\geq3$.
	
	\begin{theorem}
		Let $\mathbb{F}$ be any field, $d\geq 3$ and $f$ a polar Morse function on a combinatorial $\mathbb{F}$-orientable $d$-manifold $M$ such that the number of critical points of $f$ (counted with multiplicity) satisfies
				\begin{equation*}
			\mu_{d-i}(f; \mathbb{F})=\mu_i(f; \mathbb{F})=\left\{ 
			\begin{array}{ll}
			0&\text{for even $2\leq i\leq \lfloor\frac{d}{2}\rfloor$}\\ 
			k_i&\text{for odd $1\leq i\leq \lfloor\frac{d}{2}\rfloor$}\\ 
			\end{array}
			\right.,
		\end{equation*}
				where $k_i\geq 0$ for arbitrary $d$ and moreover $k_{\lfloor d/2\rfloor}=k_{\lceil d/2 \rceil}=0$, if $d$ is odd. Then $f$ is a tight function.
		\label{thm:MorseTight}
	\end{theorem}

	\begin{proof}
		Note that as $f$ is polar, $M$ necessarily is connected and orientable. If $d=3$, $\mu_0=\mu_3=1$ and $\mu_1=\mu_2=0$, and the statement follows immediately. Thus, let us only consider the case $d\geq 4$ from now on. 
		Assume that the vertices $v_1,\dots,v_n$ of $M$ are ordered by their $f$-values, $f(v_1)< f(v_2)<\dots<f(v_n)$. In the long exact sequence for the relative homology
				\begin{equation}
			\begin{array}{l}
			\ldots \ra H_{i+1}(M_v,M_v\backslash\set{v}) \ra H_{i}(M_v\backslash\set{v}) \overset{\iota_{i}^{*}}{\ra} H_{i}(M_v) \ra \\  \ra H_{i}(M_v,M_v\backslash\set{v}) \ra H_{i-1}(M_v\backslash\set{v}) \ra \ldots
			\end{array}
			\label{eq:rellongexactpolar}
		\end{equation}
	the tightness of $f$ is equivalent to the injectivity of the inclusion map $\iota_{i}^{*}$ for all $i$ and all $v\in V(M)$. The injectivity of $\iota_{i}^{*}$ means that for any fixed $j=1,\dots,n$, the homology $H_i(M_{v_j},M_{v_{j-1}})$ (where $M_{v_{0}}=\emptyset$) persists up to the maximal level $H_i(M_{v_n})=H_i(M)$ and is mapped injectively from level $v_j$ to level $v_{j+1}$. This obviously is equivalent to the condition for tightness given in Definition~\ref{def:tighthom}. Thus, tight triangulations can also be interpreted as triangulations with the maximal persistence of the homology in all dimensions with respect to the vertex ordering induced by $f$ (see \cite{Edelsbrunner08PersistHomSurv}). Hence, showing the tightness of $f$ is equivalent to proving the injectivity of $\iota_{i}^{*}$ at all vertices $v\in V(M)$ and for all $i$, what will be done in the following. Note that for all values of $i$ for which $\mu_i=0$, nothing has to be shown so that we only have to deal with the cases where $\mu_i>0$ below.
		
		The restriction of the number of critical points being non-zero only in every second dimension results in 
			\begin{equation*}
			\dim_{\mathbb{F}} H_{i}(M_v,M_v\backslash\set{v})\leq\mu_i(f; \mathbb{F})=0
			\end{equation*}
		and
		\begin{equation*}
			\dim_{\mathbb{F}} H_{d-i}(M_v,M_v\backslash\set{v})\leq\mu_{d-i}(f; \mathbb{F})=0
		\end{equation*}
	and thus in $H_{i}(M_v,M_v\backslash\set{v})=H_{d-i}(M_v,M_v\backslash\set{v})=0$ for all even $2\leq i\leq \lfloor \frac{d}{2} \rfloor$ and all $v\in V(M)$, as $M$ is $\mathbb{F}$-orientable.
				This implies a splitting of the long exact sequence (\ref{eq:rellongexactpolar}) at every second dimension, yielding exact sequences of the forms
				\begin{equation*}
			0 \ra  H_{i-1}(M_v\backslash\set{v}) \overset{\iota_{i-1}^{*}}{\rightarrow} H_{i-1}(M_v) \ra  H_{i-1}(M_v,M_v\backslash\set{v}) \ra \ldots
			\end{equation*}
		and
				\begin{equation*}
						0 \ra  H_{d-i-1}(M_v\backslash\set{v}) \overset{\iota_{d-i-1}^{*}}{\rightarrow} H_{d-i-1}(M_v) \ra  H_{d-i-1}(M_v,M_v\backslash\set{v}) \ra \ldots,
						\end{equation*}	
					where the inclusions $\iota_{i-1}^{*}$ and $\iota_{d-i-1}^{*}$ are injective for all vertices $v\in V(M)$, again for all even $2\leq i\leq \lfloor \frac{d}{2} \rfloor$. Note in particular, that $\mu_{d-2}=0$ always holds.
				For critical points of index $d-1$, the situation looks alike:
				\begin{equation*}
			\begin{array}{l}
			0\ra \underbrace{H_{d}(M_v\backslash\set{v})}_{=0} \ra H_{d}(M_{v}) \ra H_{d}(M_v,M_v\backslash\set{v}) \ra \\ \ra H_{d-1}(M_v\backslash\set{v}) \overset{\iota_{d-1}^{*}}{\ra} H_{d-1}(M_v)\ra  H_{d-1}(M_v,M_v\backslash\set{v}) \ra \ldots 
			\end{array}
		\end{equation*}
				By assumption, $f$ only has one maximal vertex as it is polar. Then, if $v$ is not the maximal vertex with respect to $f$, $H_{d}(M_v,M_v\backslash\set{v})=0$ and thus $\iota_{d-1}^{*}$ is injective.
				If, on the other hand, $v$ is the maximal vertex with respect to $f$, one has
				\begin{equation*}
			H_{d}(M)\iso H_{d}(M_v,M_v\backslash\set{v}),
		\end{equation*}
		as $M_v=M$ in this case. Consequently, by the exactness of the sequence above, $\iota_{d-1}^{*}$ is also injective in this case.
				Altogether it follows that $\iota_{i}^{*}$ is injective for all $i$ and for all vertices $v\in V(M)$ and thus that $f$ is $\mathbb{F}$-tight. 
	\end{proof}
	
	As we will see in Section~\ref{sec:TightnessKd}, this is a condition that can be translated into a purely combinatorial one. Examples of manifolds to which Theorem~\ref{thm:MorseTight} applies will be given in the following sections.
	
	\begin{remark}\hfill
		\begin{enumerate}[(i)]
		
		\item Theorem~\ref{thm:MorseTight} can be understood as a combinatorial equivalent of Morse's \emph{lacunary principle} \cite[Lecture~2]{Bott82LectMorseTheory}. The lacunary principle in the smooth case states that if $f$ is a smooth Morse function on a smooth manifold $M$, such that its Morse polynomial $M_t(f)$ contains no consecutive powers of $t$, then $f$ is a perfect Morse function.
		
		\item Due to the Morse relations, Theorem~\ref{thm:MorseTight} puts a restriction on the topology of manifolds admitting these kinds of Morse functions. In particular, these must have vanishing Betti numbers in the dimensions where the number of critical points is zero. Note that in dimension $d=3$ the theorem thus only holds for homology $3$-spheres with $\beta_1=\beta_2=0$ and no statements concerning the tightness of triangulations with $\beta_1>0$ can be made. One way of proving the tightness of a $2$-neighborly combinatorial $3$-manifold $M$ would be to show that the mapping
				\begin{equation}
			H_2(M_v)\rightarrow H_2(M_,M_v\backslash \set{v})
			\label{eq:surmaph2}
		\end{equation}
	is surjective for all $v\in V(M)$ and all rsl functions $f$. This would result in an injective mapping in the homology group $H_1(M_v\backslash\set{v})\ra H_1(M_v)$ for all $v\in V(M)$ -- as above by virtue of the long exact sequence for the relative homology -- and thus in the $1$-tightness of $M$, which is equivalent to the ($\mathbb{F}_2$-)tightness of $M$ for $d=3$, see \cite[Prop.~3.18]{Kuehnel95TightPolySubm}.	Unfortunately, there does not seem to be an easy to check combinatorial condition on $M$ that is sufficient for the surjectivity of the mapping (\ref{eq:surmaph2}) for all $v$ and all $f$, in contrast to the case of a combinatorial condition for the $0$-tightness of $M$ for which this is just the $2$-neighborliness of $M$.

		\end{enumerate}
		\label{rem:lacunaryprinciple}
	\end{remark}

\section{Tightness of members of $\mathcal{K}(d)$}
\label{sec:TightnessKd}

	In this section we will investigate the tightness of members of Walkup's class $\mathcal{K}(d)$, the family of all combinatorial $d$-manifolds that only have stacked $(d-1)$-spheres as vertex links. For $d\leq 2$, $\mathcal{K}(d)$ is the set of all triangulated $d$-manifolds. Kalai \cite{Kalai87RigidityLBT} showed that the stacking-condition of the links puts a rather strong topological restriction on the members of $\mathcal{K}(d)$:  
	
	\begin{theorem}[Kalai, \cite{Kalai87RigidityLBT, Bagchi08OnWalkupKd}]
		Let $d\geq 4$. Then $M$ is a connected member of $\mathcal{K}(d)$ if and only if $M$ is obtained from
		a stacked $d$-sphere by $\beta_1(M)$ combinatorial handle additions.
		\label{thm:KalaiKdConnected}
	\end{theorem}
	
	Here, a \emph{combinatorial handle addition} to a complex $C$ is defined as usual (see \cite{Walkup70LBC34Mnf, Kalai87RigidityLBT, Lutz08FVec3Mnf}) as the complex $C^{\psi}$ obtained from $C$ by identifying two facets $\Delta_1$ and $\Delta_2$ of $C$ such that $v\in V(\Delta_1)$ is identified with $w\in \Delta_2$ only if $\operatorname{d}(v,w)\geq 3$, where $V(X)$ denotes the vertex set of a simplex $X$ and $\operatorname{d}(v,w)$ the distance of the vertices $v$ and $w$ in the $1$-skeleton of $C$ seen as undirected graph (cf. \cite{Bagchi08MinTrigSphereBundCirc}).
	
	In other words, Kalai's theorem states that any connected $M\in\mathcal{K}(d)$ is necessarily homeomorphic to a connected sum with summands of the form $S^1\times S^{d-1}$ and $S^1\dtimes S^{d-1}$, compare \cite{Lutz08FVec3Mnf}. Looking at $2$-neighborly members of $\mathcal{K}(d)$, the following observation concerning the embedding of the triangulation can be made.
	
	\begin{theorem}
		Let $d=2$ or $d\geq 4$. Then any $2$-neighborly member of $\mathcal{K}(d)$ yields a tight triangulation of the underlying PL manifold.
		\label{thm:Kd2NeighborlyTight}
	\end{theorem}
	
	Note that since any triangulated $1$-sphere is stacked, $\mathcal{K}(2)$ is the set of all triangulated surfaces and that any $2$-neighborly triangulation of a surface is tight. The two conditions of the manifold being $2$-neighborly and having only stacked spheres as vertex links are rather strong as the only stacked sphere that is $k$-neighborly, $k\geq 2$, is the boundary of the simplex, see also Remark~\ref{rem:stackedneigh}. Thus, the only $k$-neighborly member of $\mathcal{K}(d)$, $k\geq 3$, $d\geq 2$, is the boundary of the $(d+1)$-simplex.
	
	The following lemma will be needed for the proof of Theorem~\ref{thm:Kd2NeighborlyTight}.

	\begin{lemma}
		Let $S$ be a stacked $d$-sphere, $d\geq 3$, and $V'\subseteq V(S)$.
		Then $H_{d-j}(\Span_S(V'))=0$ for $2\leq j \leq d-1$, where $H_{*}$ denotes the simplicial 	homology groups.
		\label{lem:StackedSphereHomology}
	\end{lemma}
	
	\begin{proof}
		Assume that $S_0=\partial \Delta^{d+1}$ and assume $S_{i+1}$ to be obtained from $S_i$ by a single stacking operation such that there exists an $N\in\mathbb{N}$ with $S_N=S$.
				Then $S_{i+1}$ is obtained from $S_i$ by removing a facet of $S_i$ and the boundary of a new $d$-simplex
		$T_i$ followed by a gluing operation of $S_i$ and $T_i$ along the  boundaries of the removed facets. This process can also be understood in terms of a bistellar $0$-move carried out on a facet of $S_i$. 
				Since this process does not remove any $(d-1)$-simplices from $S_i$ or $T_i$ we have
		$\skel_{d-1}(S_i)\subset \skel_{d-1}(S_{i+1})$.
		
				We prove the statement by induction on $i$. Clearly, the statement is true for $i=0$, as $S_0=\partial \Delta^{d+1}$ and $\partial\Delta^{d+1}$ is $(d+1)$-neighborly.		
		Now assume that the statement holds for $S_i$ and let $V_{i+1}'\subset V(S_{i+1})$. In the following we can consider the connected components $C_k$ of $\Span_{S_{i+1}}( V_{i+1}')$ separately.
		If $C_k\subset S_i$ or $C_k \subset T_i$ then the statement is true by assumption and the $(d+1)$-neighborliness of $\partial \Delta^{d+1}$, respectively. Otherwise let
		$P_1:=C_k\cap S_{i}\neq \emptyset$ and $P_2:=C_k\cap T_i\neq \emptyset$. Then
				\begin{equation*}
			H_{d-j}(P_1)\iso H_{d-j}(P_1\cap T_i)\text{ and }H_{d-j}(P_2)\iso H_{d-j}(P_2\cap S_i).
		\end{equation*}
				This yields
				\begin{equation*}
		\begin{array}{l@{}l@{}l}
		H_{d-j}(P_1\cup P_2)
		&=&H_{d-j}((P_1\cup P_2)\cap S_i\cap T_i)\\ 
		&=&H_{d-j}(\Span_{S_i\cap T_i}( V_{i+1}') )\\ 
		&=&H_{d-j}(\Span_{S_i\cap T_i}( V_{i+1}' \cap V(S_i\cap T_i)))\\ 
		&=&0,
		\end{array}
		\end{equation*}
				as $S_i\cap T_i=\partial\Delta^{d}$ which is $(d-1)$-neighborly such that the span of any vertex set has vanishing $(d-j)$-th homology for $2\leq j\leq d-1$.
	\end{proof}
	
	\bigskip

	\begin{proof}[of Theorem~\ref{thm:Kd2NeighborlyTight}]
		For $d=2$, see \cite{Kuehnel95TightPolySubm} for a proof. From now on assume that $d\geq 4$. As can be shown via excision, if $M$ is a combinatorial $d$-manifold,  $f: M\ra \R$ an rsl function on $M$ and $v\in V(M)$, then
				\begin{equation*}
			H_{*}(M_v,M_v\backslash\set{v})\iso H_{*}(M_v\cap \str(v), M_v\cap \lk(v)).
		\end{equation*}
				Now let $d\geq 4$, $1<i<d-1$. The long exact sequence for the relative homology
				\begin{equation*}
			\begin{array}{l}
				\dots\ra H_{d-i}(M_v\cap \str(v))\ra H_{d-i}(M_v\cap \str(v), M_v\cap \lk(v))\ra\\ 
				\ra H_{d-i-1}(M_v\cap \lk(v))\ra H_{d-i-1}(M_v\cap \str(v))\ra\dots
			\end{array}
		\end{equation*}
				yields an isomorphism 
				\begin{equation}
			H_{d-i}(M_v\cap \str(v), M_v\cap \lk(v))\iso H_{d-i-1}(M_v\cap \lk(v)),
			\label{eq:IsomorphyHomologyStarLink}
		\end{equation}
	as $M_v\cap \str(v)$ is a cone over $M_v\cap \lk(v)$, thus contractible and we have $H_{d-i}(M_v\cap \str(v))=H_{d-i-1}(M_v\cap \str(v))=0$.
				
		Since $M\in\mathcal{K}(d)$, all vertex links in $M$ are stacked $(d-1)$-spheres and thus Lemma~\ref{lem:StackedSphereHomology} applies to the right hand side of (\ref{eq:IsomorphyHomologyStarLink}). This implies that a $d$-manifold $M\in \mathcal{K}(d)$, $d\geq 4$, cannot have critical points of index $2\leq i \leq d-2$, i.e. 
		$\mu_2(f; \mathbb{F})=\dots=\mu_{d-2}(f; \mathbb{F})=0$.
		
		Furthermore, the $2$-neighborliness of $M$ implies that any rsl function on $M$ is polar. Thus, all prerequisites of Theorem~\ref{thm:MorseTight} are fulfilled, $f$ is tight and consequently $M$ is a tight triangulation, what was to be shown.
	\end{proof}

	\begin{remark}
		In even dimensions $d\geq 4$, Theorem~\ref{thm:Kd2NeighborlyTight} can also be proved without using Theorem~\ref{thm:MorseTight}. In this case the statement follows from the $2$-neighborliness of $M$ (that yields $\mu_0(f; \mathbb{F})=\beta_0$ and $\mu_d(f; \mathbb{F})=\beta_d$), and the Morse relations \ref{thm:MorseRelations} which then yield $\mu_1(f; \mathbb{F})=\beta_1$ and $\mu_{d-1}(f; \mathbb{F})=\beta_{d-1}$ for any rsl function $f$, as $\mu_2(f; \mathbb{F})=\dots=\mu_{d-2}(f; \mathbb{F})=0$.
	\end{remark}

	As a consequence, the stacking condition of the links already implies the vanishing of $\beta_2,\dots,\beta_{d-2}$ (as by the Morse relations $\mu_i\geq\beta_i$), in accordance with Kalai's Theorem~\ref{thm:KalaiKdConnected}.

	An example of a series of tight combinatorial manifolds is the infinite series of sphere bundles over the circle due to K\"uhnel \cite{Kuehnel86HigherDimCsaszar}. The triangulations in this series are all $2$-neighborly on $f_0=2d+3$ vertices. They are homeomorphic to $S^{d-1}\times S^1$ in even dimensions and to $S^{d-1}\dtimes S^1$ in odd dimensions. Furthermore, all links are stacked and thus Theorem~\ref{thm:Kd2NeighborlyTight} applies providing an alternative proof of the tightness of the triangulations in this series.	
	
	\begin{corollary}
		All members $M^d$ of the series of triangulations in \cite{Kuehnel86HigherDimCsaszar} are $2$-neighborly and lie in the class $\mathcal{K}(d)$. They are thus tight triangulations by Theorem~\ref{thm:Kd2NeighborlyTight}.
	\end{corollary}
		
 	Another example of a triangulation to which Theorem~\ref{thm:Kd2NeighborlyTight} applies is due to Bagchi and Datta \cite{Bagchi08OnWalkupKd}. It is an example of a so called \emph{tight-neighborly} triangulation as defined by Lutz, Sulanke and Swartz \cite{Lutz08FVec3Mnf}. For this class of manifolds, Theorem~\ref{thm:Kd2NeighborlyTight} holds for $d=2$ and $d\geq 4$. Tight-neighborly triangulations will be described in more detail in the next section.

\section{Tight-neighborly triangulations}
\label{sec:TightNeigh}

	Beside the class of combinatorial $d$-manifolds with stacked spheres as vertex links $\mathcal{K}(d)$, Walkup \cite{Walkup70LBC34Mnf} also defined the class $\mathcal{H}(d)$. This is the family of all simplicial complexes that can be obtained from the boundary complex of the $(d+1)$-simplex by a series of zero or more of the following three operations: (i) stellar subdivision of facets, (ii) combinatorial handle additions and (iii) forming connected sums of objects obtained from the first two operations.
	
	The two classes are closely related. Obviously, the relation $\mathcal{H}(d)\subset \mathcal{K}(d)$ holds. Kalai \cite{Kalai87RigidityLBT} showed the reverse inclusion $\mathcal{K}(d)\subset \mathcal{H}(d)$ for $d\geq 4$. 
	
	Note that the condition of the $2$-neighborliness of an $M\in \mathcal{K}(d)$ in Theorem~\ref{thm:Kd2NeighborlyTight} is equivalent to the first Betti number $\beta_1(M)$ being maximal with respect to the vertex number $f_0(M)$ of $M$ (as a $2$-neighborly triangulation does not allow any  handle additions). Such manifolds are exactly the cases of equality in \cite[Th.~5.2]{Novik08SocBuchsMod}. In their recent work \cite{Lutz08FVec3Mnf}, Lutz, Sulanke and Swartz prove the following\footnote{The author would like to thank Frank Lutz for fruitful discussions about tight-neighborly triangulations and pointing him to the work \cite{Lutz08FVec3Mnf} in the first place.} 
	
	\begin{theorem}[Theorem $5$ in \cite{Lutz08FVec3Mnf}]
		Let $\mathbb{K}$ be any field and let $M$ be a $\mathbb{K}$-orientable triangulated $d$-manifold with $d\geq 3$. Then
				\begin{equation}
			f_0(M)\geq \left\lceil \frac{1}{2}\left(2d+3+\sqrt{1+4(d+1)(d+2)\beta_1(M;\mathbb{K})} \right)	\right\rceil.
			\label{eq:TightNeighEq}
		\end{equation}
	\end{theorem}

	\begin{remark}
		As pointed out in \cite{Lutz08FVec3Mnf}, for $d=2$ inequality (\ref{eq:TightNeighEq}) coincides with Heawood's inequality 
				\begin{equation*}
			f_0(M)\geq \left\lceil \frac{1}{2}\left(7+\sqrt{49-24\chi(M)}\right)\right\rceil,
		\end{equation*}
				if one replaces $\beta_1(M;\mathbb{K})$ by $\frac{1}{2}\beta_1(M;\mathbb{K})$ to account for the double counting of the middle Betti number $\beta_1(M;\mathbb{K})$ of surfaces by Poincar\'{e} duality. Inequality (\ref{eq:TightNeighEq}) can also be written in the form
				\begin{equation*}
			\binom{f_0 -d -1}{2}\geq \binom{d+2}{2}\beta_1.  
		\end{equation*}
				Thus, Theorem $5$ in \cite{Lutz08FVec3Mnf} settles K\"uhnel's conjectured bounds
				\begin{equation*}
					\binom{f_0-d+j-2}{j+1}\geq \binom{d+2}{j+1} \beta_j\quad\text{with}\quad 1\leq j\leq \lfloor \frac{d-1}{2} \rfloor 
		\end{equation*}
				in the case $j=1$.
	\end{remark}

	For $\beta_1=1$, the bound (\ref{eq:TightNeighEq}) coincides with the Brehm-K\"uhnel bound $f_0\geq 2d+4-j$ for $(j-1)$-connected but not $j$-connected $d$-manifolds in the case $j=1$, see \cite{Brehm87CombMnfFewVert}. Inequality (\ref{eq:TightNeighEq}) is sharp by the series of vertex minimal triangulations of sphere bundles over the circle presented in \cite{Kuehnel86HigherDimCsaszar}.
	
	Triangulations of connected sums of sphere bundles $(S^{2}\times S^1)^{\#k}$ and $(S^{2}\dtimes S^1)^{\#k}$ attaining equality in (\ref{eq:TightNeighEq}) for $d=3$ were discussed in \cite{Lutz08FVec3Mnf}. Note that such triangulations are necessarily $2$-neighborly.
	
	\begin{definition}[tight-neighborly triangulation, \cite{Lutz08FVec3Mnf}]
		Let $d\geq 2$ and let $M$ be a triangulation of $(S^{d-1}\times S^1)^{\#k}$ or $(S^{d-1}\dtimes S^1)^{\#k}$ attaining equality in (\ref{eq:TightNeighEq}). Then $M$ is called a \emph{tight-neighborly} triangulation.
	\end{definition}
	
	For $d\geq 4$, all triangulations of  $\mathbb{F}$-orientable $\mathbb{F}$-homology $d$-manifolds with equality in (\ref{eq:TightNeighEq}) lie in $\mathcal{H}(d)$ and are tight-neighborly triangulations of $(S^{d-1}\times S^1)^{\#k}$ or $(S^{d-1}\dtimes S^1)^{\#k}$ by Theorem 5.2 in \cite{Novik08SocBuchsMod}.
	
	The authors conjectured \cite[Conj.~13]{Lutz08FVec3Mnf} that all tight-neighborly triangulations are tight in the classical sense of Definition~\ref{def:BasicTightComb} and showed that the conjecture holds in the following cases: for $\beta_1=0,1$ and any $d$ and for $d=2$ and any $\beta_1$. Indeed, the conjecture also holds for any $d\geq 4$ and any $\beta_1$ as a direct consequence of Theorem~\ref{thm:Kd2NeighborlyTight}.
	
	\begin{corollary}
		For $d\geq 4$, all tight-neighborly triangulations are tight. 
	\end{corollary}
	
	\begin{proof}
		For $d\geq 4$, one has $\mathcal{H}(d)=\mathcal{K}(d)$ and the statement is true for all $2$-neighborly members of $\mathcal{K}(d)$ by Theorem~\ref{thm:Kd2NeighborlyTight}.
	\end{proof}
	
	It remains to be investigated, whether for vertex minimal triangulations of $d$-handlebodies, $d\geq 3$, the reverse implication is true, too, i.e.\  that for this class of triangulations the terms of tightness and tight-neighborliness are equivalent.
 
	\begin{question}
	 	Let $d\geq 4$ and let $M$ be a tight triangulation homeomorphic to $(S^{d-1}\times S^{1})^{\#k}$ or $(S^{d-1}\dtimes S^{1})^{\#k}$. Does this imply that $M$ is tight-neighborly?
	\end{question}
	
	As was shown in \cite{Lutz08FVec3Mnf}, at least for values of $\beta_1=0,1$ and any $d$ and for $d=2$ and any $\beta_1$ this is true.
	
	\bigskip
	
	One example of a triangulation for which Theorem~\ref{thm:Kd2NeighborlyTight} holds is due to Bagchi and Datta \cite{Bagchi08OnWalkupKd}. The triangulation $M^4_{15}$ of $(S^3\dtimes S^1)^{\#3}$ from \cite{Bagchi08OnWalkupKd} is a $2$-neighborly combinatorial $4$-manifold on $15$ vertices that is a member of $\mathcal{K}(4)$ with $f$-vector $f=(15,\,105,\,230,\,240,\,96)$. Since  $M^4_{15}$ is tight-neighborly, we have the following corollary.
	
	\begin{corollary}
		The $4$-manifold $M^4_{15}$ given in \cite{Bagchi08OnWalkupKd} is a tight triangulation.
	\end{corollary}
	
	The next possible triples of values of $\beta_1$, $d$ and $n$ for which a 2-neighborly member of $\mathcal{K}(d)$ could exist (compare \cite{Lutz08FVec3Mnf}) are listed in Table~\ref{tab:PossTriplesKd}. Apart from the sporadic examples in dimension $4$ and the infinite series of higher dimensional analogues of Cs{\'a}sz{\'a}r's torus in arbitrary dimension $d\geq 2$ due to Kühnel \cite{Kuehnel86HigherDimCsaszar}, cf. \cite{Kuehnel96PermDiffCyc, Bagchi08MinTrigSphereBundCirc, Chestnut08EnumPropTrigSpherBund}, mentioned earlier, no further examples are known as of today.

	\begin{table}
		\caption{Known and open cases for $\beta_1$, $d$ and $n$ of $2$-neighborly members of $\mathcal{K}(d)$.}
		\label{tab:PossTriplesKd}

		\centering
		\begin{tabular}{l|l|l|l|l}
			$\beta_1$&$d$&$n$&top.\ type&reference\\ \hline
			
			$0$&any $d$&$d+1$&$S^{d-1}$&$\partial \Delta^d$\\
			$1$&any even $d\geq 2$&$2d+3$&$S^{d-1}\times S^1$&\cite{Kuehnel86HigherDimCsaszar} ($d=2$: \cite{Moebius86Werke, Csaszar49PolyWithoutDiags})\\
			$1$&any odd $d\geq 2$&$2d+3$&$S^{d-1}\dtimes S^1$&\cite{Kuehnel86HigherDimCsaszar} ($d=3$: \cite{Walkup70LBC34Mnf, Altshuler74NeighComb3Mnf9Vert})\\
			$2$&$13$&$35$&?&\\
			$3$&$4$&$15$&$(S^3\dtimes S^1)^{\#3}$&\cite{Bagchi08OnWalkupKd}\\ 
			$5$&$5$&$21$&?&\\ 
			$8$&$10$&$44$&?&\\ 
		\end{tabular}
	\end{table}
	
	Especially in (the odd) dimension $d=3$, things seem to be a bit more subtle, as already laid out in Remark~\ref{rem:lacunaryprinciple}.
	As Altshuler and Steinberg \cite{Altshuler73Neigh4Poly9Vert} showed that the link of any vertex in a neighborly $4$-polytope is stacked (compare also Remark~$5$ in \cite{Kalai87RigidityLBT}), we know that the class $\mathcal{K}(3)$ is rather big compared to $\mathcal{H}(3)$. Thus, a statement equivalent to Theorem~\ref{thm:Kd2NeighborlyTight} is not surprisingly false for members of $\mathcal{K}(3)$, a counterexample being the boundary of the cyclic polytope $\partial C(4,6)\in \mathcal{K}(3)$ which is $2$-neighborly but certainly not a tight triangulation as it has empty triangles.
	The only currently known non-trivial example of a tight-neighborly combinatorial $3$-manifold is a $9$-vertex triangulation $M^3$ of $S^2\dtimes S^1$, independently found by Walkup \cite{Walkup70LBC34Mnf} and Altshuler and Steinberg \cite{Altshuler74NeighComb3Mnf9Vert}. This triangulation is combinatorially unique, as was shown by Bagchi and Datta \cite{Bagchi08UniqWalkup9V3DimKleinBottle}. For $d=3$, it is open whether there exist tight-neighborly triangulations for higher values of $\beta_1\geq 2$, see \cite[Question 12]{Lutz08FVec3Mnf}.
	
	The fact that $M^3$ is a tight triangulation is well known, see \cite{Kuehnel95TightPolySubm}. Yet, we will present here another proof of the tightness of $M^3$. It is a rather easy procedure when looking at the $4$-polytope $P$ the boundary of which $M^3$ was constructed from by one elementary combinatorial handle addition, see also \cite{Bagchi08OnWalkupKd}.
	
	\begin{lemma}
		Walkup's $9$-vertex triangulation $M^3$ of $S^2\dtimes S^1$ is tight.
	\end{lemma}

	\begin{proof}
		Take the stacked $4$-polytope $P$ with $f$-vector $f(P)=(13, 42, 58, 37, 9)$ from \cite{Walkup70LBC34Mnf}. Its facets are
				\begin{center}
			\begin{tabular}{lll}
			$\langle 1\,2\,3\,4\,5 \rangle$, &
			$\langle 2\,3\,4\,5\,6 \rangle$, &
			$\langle 3\,4\,5\,6\,7 \rangle$,\\ 
			
			$\langle 4\,5\,6\,7\,8 \rangle$, &
			$\langle 5\,6\,7\,8\,9 \rangle$, &
			$\langle 6\,7\,8\,9\,10 \rangle$,\\ 
			
			$\langle 7\,8\,9\,10\,11 \rangle$, &
			$\langle 8\,9\,10\,11\,12 \rangle$, &
			$\langle 9\,10\,11\,12\,13 \rangle$.
			\end{tabular}
		\end{center}
				As $P$ is stacked it has missing edges (called \emph{diagonals}), but no empty faces of higher dimension. 
		
		Take the boundary $\partial P$ of $P$. By construction, $P$ has no inner $i$-faces for $0\leq i\leq 2$, so that $\partial P$ has the 36 diagonals of $P$ and additionally $8$ empty tetrahedra, but no empty triangles. As $\partial P$ is a $3$-sphere, the empty tetrahedra are all homologous to zero.  
		
		Now form a $1$-handle over $\partial P$ by removing the two tetrahedra $\langle 1,2,3,4\rangle$ and $\langle 10,11,12,13\rangle$ from $\partial P$ followed by an identification of the four vertex pairs $(i, i+9)$, $1\leq i \leq 4$, where the newly identified vertices are labeled with $1,\dots,4$.
		
		This process yields a $2$-neighborly combinatorial manifold $M^3$ with $1$3$-4=9$ vertices and one additional empty tetrahedron $\langle 1,2,3,4\rangle$, which is the generator of $H_2(M)$.
		
		As $M^3$ is $2$-neighborly, it is $0$-tight and as $\partial P$ had no empty triangles, two empty triangles in the span of any vertex subset $V'\subset V(M)$ are always homologous. Thus, $M^3$ is a tight triangulation. 
	\end{proof}
	
	The construction in the proof above could probably be used in the general case with $d=3$ and $\beta_1\geq 2$: one starts with a stacked $3$-sphere $M_0$ as the boundary of a stacked $4$-polytope which by construction does not contain empty $2$-faces and then successively forms handles over this boundary $3$-sphere (obtaining triangulated manifolds $M_1,\dots,M_n=M$) until the resulting triangulation $M$ is $2$-neighborly and fulfills equality in (\ref{eq:TightNeighEq}). Note that this can only be done in the regular cases of (\ref{eq:TightNeighEq}), i.e. where (\ref{eq:TightNeighEq}) admits integer solutions for the case of equality. For a list of possible configurations see \cite{Lutz08FVec3Mnf}.

\section{$k$-stacked spheres and the class $\mathcal{K}^k(d)$}
\label{sec:Kkd}

	McMullen and Walkup \cite{McMullen71GeneralizedLBC} extended the notion of stacked polytopes to \emph{$k$-stacked polytopes} as simplicial $d$-polytopes that can be triangulated without introducing new $j$-faces for $0\leq j\leq d-k-1$. 
	
	\begin{definition}[$k$-stacked balls and spheres, \cite{McMullen71GeneralizedLBC, Kalai87RigidityLBT}]\hfill
		A \emph{$k$-stacked $(d+1)$-ball}, $0\leq k\leq d$, is a triangulated $(d+1)$-ball that has no interior $j$-faces, $0\leq j\leq d-k$. A \emph{minimally $k$-stacked $(d+1)$-ball} is a $k$-stacked $(d+1)$-ball that is not $(k-1)$-stacked. The boundary of any (minimally) $k$-stacked $(d+1)$-ball is called a \emph{(minimally) $k$-stacked $d$-sphere}.
		\label{def:kStackedSphere}
	\end{definition}
	
	Note that in this context the ordinary stacked $d$-spheres are exactly the $1$-stacked $d$-spheres. Note furthermore, that a $k$-stacked $d$-sphere is obviously also $(k+l)$-stacked, where $l\in\mathbb{N}$, $k+l\leq d$, compare \cite{Bagchi08LBTNormPseudoMnf}. The simplex $\Delta^{d+1}$ is the only $0$-stacked $(d+1)$-ball and the boundary of the simplex $\partial\Delta^{d+1}$ is the only $0$-stacked $d$-sphere. Keep in mind that all triangulated $d$-spheres are at least $d$-stacked \cite[Rem.~9.1]{Bagchi08LBTNormPseudoMnf}.
	
	\begin{figure}		\centering
		\includegraphics[height=4cm]{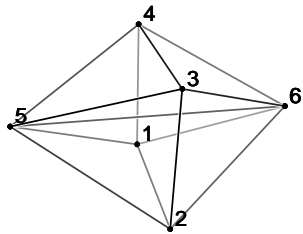}
		\caption{A minimally $2$-stacked $S^2$ as the boundary complex of a subdivided $3$-octahedron.}
		\label{fig:2stackeds2}
	\end{figure}
	
	Figure~\ref{fig:2stackeds2} shows the boundary of an octahedron as an example of a minimally $2$-stacked $2$-sphere $S$ with $6$ vertices. The octahedron that is subdivided along the inner diagonal $(5,6)$ can be regarded as a triangulated $3$-ball $B$ with $\skel_0(S)=\skel_0(B)$ and $\partial B=S$. Note that although all vertices of $B$ are on the boundary, there is an inner edge such that the boundary is $2$-stacked, but not $1$-stacked. In higher dimensions, examples of minimally $d$-stacked $d$-spheres exist as boundary complexes of subdivided $d$-cross polytopes with an inner diagonal. 

	Akin to the $1$-stacked case, a more geometrical characterization of $k$-stacked $d$-spheres can be given via bistellar moves (also known as Pachner moves, see \cite{Pachner87KonstrMethKombHomeo}), at least for $k\leq \lceil \frac{d}{2}\rceil$. 

	\begin{definition}[bistellar moves]
		Let $M$ be a triangulated $d$-manifold and let $A$ be a $(d-i)$-face of $M$, $0\leq i\leq d$, such that there exists an $i$-simplex $B$ that is not a face of $M$ with $\lk_M(A)=\partial B$. Then a \emph{bistellar $i$-move}  $\Phi_A$ on $M$ is defined by
		\begin{equation*}
			\Phi_A(M):=\left(M\backslash (A*\partial B)\right)\cup (\partial A * B),
		\end{equation*}
		where $*$ denotes the join operation for simplicial complexes. Bistellar $i$-moves with $i> \lfloor\frac{d}{2}\rfloor$ are also-called \emph{reverse $(d-i)$-moves}. 
	\end{definition}

	See Figure \ref{fig:flips} for an example illustration of bistellar moves in dimension $d=3$. Note that for any bistellar move $\Phi_A(M)$, $A* B$ forms a $(d+1)$-simplex. Thus, any sequence of bistellar moves defines a sequence of $(d+1)$-simplices --- this we will call the \emph{induced sequence of $(d+1)$-simplices} in the following.	
	
	The characterization of $k$-stacked $d$-spheres using bistellar moves is the following. 
	
	\begin{lemma}
	For $k\leq \lceil \frac{d}{2}\rceil$, a complex $S$ obtained from the boundary of the $(d+1)$-simplex by a sequence of bistellar $i$-moves, $0\leq i\leq k-1$, is a $k$-stacked $d$-sphere.	
	\label{lem:kstackbistellar}	
	\end{lemma}
	
	\begin{proof}
	As $k\leq \lceil \frac{d}{2}\rceil$, the sequence of $(d+1)$-simplices induced by the sequence of bistellar moves is duplicate free and defines a simplicial $(d+1)$-ball $B$ with $\partial B=S$. Furthermore, $\skel_{d-k}(B)=\skel_{d-k}(S)$ holds as no bistellar move in the sequence can contribute an inner $j$-face to $B$, $0\leq j\leq d-k$. Thus, $S$ is a $k$-stacked $d$-sphere.	
	\end{proof}
	
	Keep in mind though, that this interpretation does not hold for values $k > \lceil \frac{d}{2}\rceil$, as in this case the sequence of $(d+1)$-simplices induced by the sequence of bistellar moves may have duplicate entries, as opposed to the case with $k\leq \lceil \frac{d}{2}\rceil$. 
	
	In terms of bistellar moves, the minimally $2$-stacked sphere in Figure~\ref{fig:2stackeds2} can be constructed as follows: Start with a solid tetrahedron and stack another tetrahedron onto one of its facets (a $0$-move). Now introduce the inner diagonal $(5,6)$ via a bistellar $1$-move. Clearly, this complex is not bistellarly equivalent to the simplex by only applying reverse $0$-moves (and thus not ($1$-)stacked) but it is bistellarly equivalent to the simplex by solely applying reverse $0$-, and $1$-moves and thus minimally $2$-stacked.
	
	The author is one of the authors of the toolkit \texttt{simpcomp} \cite{simpcomp, simpcompISSAC} for simplicial constructions in the \texttt{GAP} system \cite{GAP4}. \texttt{simpcomp} contains a randomized algorithm that checks whether a given $d$-sphere is $k$-stacked, $k\leq \lceil \frac{d}{2}\rceil$, using the argument above.

	\begin{figure}
		\centering
			\includegraphics[height=3.2cm]{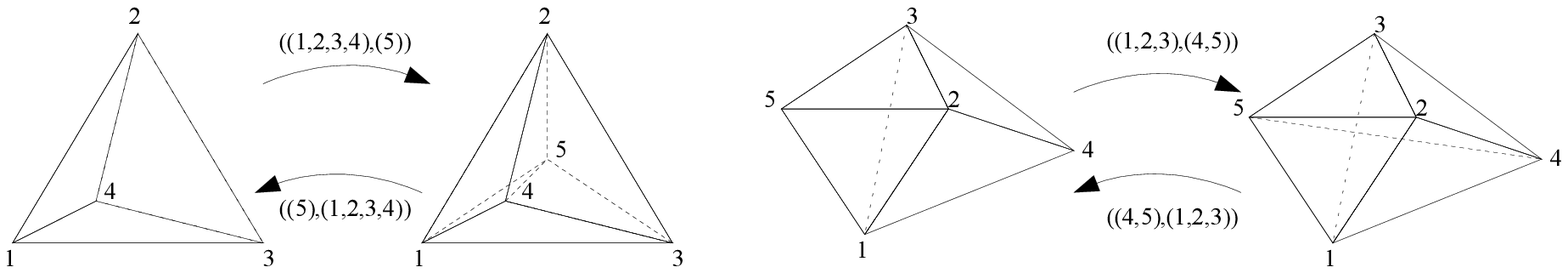}			
			\caption{Left: a $0$-move on a tetrahedron and its inverse $3$-move (i.\ e.\ a reverse $0$-move). Right: a $1$-move on two tetrahedra glued together at one triangle and its inverse $2$-move (i.e.\ a reverse $1$-move).}
            \label{fig:flips}
	\end{figure}	
		
	With the notion of $k$-stacked spheres at hand one can define the following generalization of Walkup's class $\mathcal{K}(d)$.
	
	\begin{definition}[the class $\mathcal{K}^k(d)$]
		Let $\mathcal{K}^k(d)$, $k\leq d$, be the family of all $d$-dimensional simplicial complexes all whose vertex links are $k$-stacked spheres. 
	\end{definition}
	
	Note that $\mathcal{K}^d(d)$ is the set of all triangulated manifolds for any $d$ and that Walkup's class $\mathcal{K}(d)$ coincides with $\mathcal{K}^1(d)$ above. In analogy to the $1$-stacked case, a $(k+1)$-neighborly member of $\mathcal{K}^k(d)$ with $d\geq 2k$ necessarily has vanishing $\beta_{1},\dots,\beta_{k-1}$. Thus, it seems reasonable to ask for the existence of a generalization of Kalai's Theorem~\ref{thm:KalaiKdConnected} to the class of $\mathcal{K}^k(d)$ for $k\geq 2$.

	Furthermore, one might be tempted to ask for a generalization of Theorem~\ref{thm:Kd2NeighborlyTight} to the class $\mathcal{K}^k(d)$ for $k\geq 2$. Unfortunately, there seems to be no direct way of generalizing Theorem~\ref{thm:Kd2NeighborlyTight} to also hold for members of $\mathcal{K}^k(d)$ giving a combinatorial condition for the tightness of such triangulations. The key obstruction here is the fact that a generalization of Lemma~\ref{lem:StackedSphereHomology} is impossible. While in the case of ordinary stacked spheres a bistellar $0$-move does not introduce inner simplices to the $(d-1)$-skeleton, the key argument in Lemma~\ref{lem:StackedSphereHomology}, this is not true for bistellar $i$-moves for $i\geq 1$.  

	Nonetheless, an analogous result to Theorem~\ref{thm:Kd2NeighborlyTight} should be true for such triangulations. 
	
	\begin{question}
		Let $d\geq 4$ and $2\leq k\leq \lfloor \frac{d+1}{2}\rfloor$ and let $M$ be a $(k+1)$-neighborly combinatorial manifold such that $M\in \mathcal{K}^k(d)$. Does this imply the tightness of $M$?
		\label{ques:kneightight}
	\end{question}

	\begin{remark}
		Note that all vertex links of $(k+1)$-neighborly members of $\mathcal{K}^k(d)$ are $k$-stacked, $k$-neighborly $(d-1)$-spheres. McMullen and Walkup \cite[Sect.~3]{McMullen71GeneralizedLBC} showed that there exist $k$-stacked, $k$-neighborly $(d-1)$-spheres on $n$ vertices for any $2\leq 2k\leq d< n$. Some examples of such spheres will be given in the following. The conditions of being $k$-stacked and $k$-neighborly at the same time is strong as the two conditions tend to exclude each other in the following sense: McMullen and Walkup showed that if a $d$-sphere is $k$-stacked and $k'$-neighborly with $k'>k$, then it is the boundary of the simplex. In that sense, the $k$-stacked $k$-neighborly spheres appear as the most strongly restricted non-trivial objects of this class: The conditions in Theorem~\ref{thm:Kd2NeighborlyTight} (with $k=1$) and in Question~\ref{ques:kneightight} are the most restrictive ones still admitting non-trivial solutions.
	\label{rem:stackedneigh}
		
	\end{remark}

	\begin{remark}
		Most recently, Bagchi and Datta \cite{Bagchi11StellSphereTightCritCombMnf} gave a negative answer to Question~\ref{ques:kneightight} in odd dimensions $d=2k+1$ \cite[Prop.~16]{Bagchi11StellSphereTightCritCombMnf}, but could almost prove the statement for $d\neq 2k + 1$ \cite[Prop.~20]{Bagchi11StellSphereTightCritCombMnf}.
	\end{remark}

	\begin{table}
	\caption{Some known tight triangulations and their membership in the classes $\mathcal{K}^k(d)$, cf. \cite{Kuehnel99CensusTight}, with \emph{$n$} denoting the number of vertices of the triangulation and \emph{nb.} its neighborliness.\label{tab:tighttrigkd}}

	\centering
	\begin{tabular}{l|l|l|l|l}
		{$d$}&{top.\ type}&{$n$}&{nb.}&{$k$}\\ \hline
		
		$4$&	$\mathbb{C}P^2$	&$9$	&$3$	&$2$\\ 
		$4$&	$K3$			&$16$	&$3$	&$2$\\ 
		$4$&	$(S^3\dtimes S^1)\#(\mathbb{C}P^2)^{\#5}$& $15$ &$2$ &$2$\\ 
		$5$&	$S^3\times  S^2$ &$12$   &$3$ &$2$\\ 
		$5$&	$SU(3)/SO(3)$ &$13$ &$3$ &$3$\\ 
		$6$&	$S^3\times S^3$ &$13$ &$4$ &$3$\\ 
	\end{tabular}

	\end{table}
	
	K\"uhnel and Lutz \cite{Kuehnel99CensusTight} gave an overview of the currently known tight triangulations. The statement of Question~\ref{ques:kneightight} holds for all the triangulations listed in \cite{Kuehnel99CensusTight}. Note that there even exist $k$-neighborly triangulations in $\mathcal{K}^k(d)$ that are tight and thus fail to fulfill the prerequisites of Question~\ref{ques:kneightight} (see Table~\ref{tab:tighttrigkd}).  

	Although we did not succeed in proving conditions for the tightness of triangulations lying in $\mathcal{K}^k(d)$, $k\geq 2$, these have nonetheless interesting properties that we will investigate upon in the following. Also, many known tight triangulations are members of these classes, as will be shown.	Our first observation is that the neighborliness of a triangulation is closely related to the property of being a member of $\mathcal{K}^k(d)$.
	
	\begin{lemma}
		Let $k\in \N$ and $M$ be a combinatorial $d$-manifold, $d\geq 2k$, that is a $(k+1)$-neighborly triangulation. Then $M\in \mathcal{K}^{d-k}(d)$.
		\label{lem:NeighbStacked}
	\end{lemma}
	 
	\begin{proof}
		If $M$ is $(k+1)$-neighborly, then for any $v\in V(M)$, $\lk(v)$ is $k$-neighborly.
		As $\lk(v)$ is PL homeomorphic to $\partial \Delta^{d}$ (since $M$ is a combinatorial manifold) there exists a $d$-ball $B$ with $\partial B=\lk(v)$ (cf.\ \cite{Bagchi08LBTNormPseudoMnf}). Since $\lk(v)$ is $k$-neighborly, $\skel_{k-1}(B)=\skel_{k-1}(\lk(v))$.
				By Definition~\ref{def:kStackedSphere}, the link of every vertex $v\in V(M)$ then is $(d-k)$-stacked and thus $M\in \mathcal{K}^{d-k}(d)$.
			\end{proof}

	As pointed out in Section~\ref{sec:intro}, K\"uhnel \cite[Chap.~4]{Kuehnel95TightPolySubm} investigated $(k+1)$-neighborly triangulations of  $2k$-manifolds and showed that all these are tight triangulations. By Lemma~\ref{lem:NeighbStacked}, all their vertex links are  $k$-stacked spheres.

	\begin{corollary}
		Let $M$ be a $(k+1)$-neighborly (tight) triangulation of a $2k$-manifold. Then $M$ lies in $\mathcal{K}^k(2k)$.
	\end{corollary}
	
	In particular, this holds for many vertex minimal (tight) triangulations of $4$-manifolds.

	\begin{corollary}
		The known examples of the vertex-minimal tight triangulation of a $K3$-surface with $f$-vector $f=(16, 120, 560, 720, 288)$ due to Casella and K\"uhnel \cite{Casella01TrigK3MinNumVert} and the unique vertex-minimal tight triangulation of $\C P^2$ with $f$-vector $f=(9, 36, 84, 90, 36)$ due to K\"uhnel \cite{Kuehnel83Uniq3Nb4MnfFewVert}, cf. \cite{Kuehnel83The9VertComplProjPlane} are $3$-neighborly triangulations that lie in $\mathcal{K}^2(4)$.		\label{corr:K3CP2Tight}
	\end{corollary}

	\bigskip					
		
	Let us now shed some light on properties of members of $\mathcal{K}^2(6)$. First recall that there exists a \emph{Generalized Lower Bound Conjecture} (GLBC) due to McMullen and Walkup as an extension to the classical Lower Bound Theorem for triangulated spheres as follows.
	
	\begin{conjecture}[GLBC, cf. \cite{McMullen71GeneralizedLBC, Bagchi08LBTNormPseudoMnf}]
		For $d\geq 2k+1$, the face-vector $(f_0,\dots, f_d)$ of any triangulated 
		$d$-sphere $S$ satisfies
		\begin{equation}
			f_j\geq\left\{
			\begin{array}{ll}
				\sum_{i=-1}^{k-1} (-1)^{k-i+1} \binom{j-i-1}{j-k} \binom{d-i+1}{j-i} f_i,&\quad\text{if }k\leq j\leq d-k,\\ 
				
				\sum_{i=-1}^{k-1} (-1)^{k-i+1} \left[ \binom{j-i-1}{j-k} \binom{d-i+1}{j-i}\right.&\\ 
				-\binom{k}{d-j+1} \binom{d-i}{d-k+1}&\\ 
				\left.+\sum_{l=d-j}^{k+1} (-1)^{k-l} \binom{l}{d-j} \binom{d-i}{d-l+1}\right]f_i,&\quad\text{if }d-k+1\leq j\leq d.
			\end{array}
			\right.
		\end{equation}

	Equality holds here for any $j$ if and only if $S$ is a $k$-stacked $d$-sphere.
		\label{conj:GLBC}
	\end{conjecture}
	
	The GLBC implies the following theorem for $d=6$, which is a $6$-dimensional analogue of Walkup's theorem \cite[Thm.~5]{Walkup70LBC34Mnf}, \cite[Prop.~7.2]{Kuehnel95TightPolySubm}, see also Swartz' Theorem 4.10 in \cite{Swartz08FaceEnumSpheresMnf}. 
	
	\begin{theorem}
		Assuming the validity of the Generalized Lower Bound Conjecture \ref{conj:GLBC}, for any combinatorial $6$-manifold $M$ the inequality
				\begin{equation}
			f_2(M)\geq 28\chi(M)-21 f_0 +6 f_1
			\label{eq:Ineqf2K26}
		\end{equation}
				holds. If $M$ is $2$-neighborly, then
				\begin{equation}
			f_2(M)\geq 28\chi(M)+3f_0(f_0-8)
			\label{eq:Ineqf2K26Neigh}
		\end{equation}
				holds. In either case equality is attained if and only if $M\in\mathcal{K}^2(6)$.
		\label{conj:2Neighborly6ManifoldFvec}
	\end{theorem}	
	
	\begin{proof}
		Clearly,
				\begin{equation}
			f_3(M)=\frac{1}{4}\sum_{v\in V(M)} f_2(\lk(v)).
			\label{eq:Eqf3f2}
		\end{equation}
				By applying the GLBC \ref{conj:GLBC} to all the vertex links of $M$ one obtains a lower bound on $f_2(\lk(v))$ for all $v\in V(M)$:
								\begin{equation}
			f_2(\lk(v))\geq 35 -15f_0(\lk(v)) +5 f_1(\lk(v)).
			\label{eq:Ineq2Stacked5Sphere}
		\end{equation}
				Here equality is attained if and only if $\lk(v)$ is $2$-stacked. Combining (\ref{eq:Eqf3f2}) and (\ref{eq:Ineq2Stacked5Sphere}) yields a lower bound
				\begin{equation}
			\begin{array}{l@{}l@{}l}
			f_3(M)&\geq&\frac{1}{4}\sum_{v\in V(M)} 35 -15f_0(\lk(v)) +5 f_1(\lk(v))\\ 
			&=&\frac{5}{4}\left(7 f_0(M) - 6 f_1(M) + 3 f_2(M)\right),
			\end{array}
			\label{eq:Ineq5manif3}
		\end{equation}
				for which equality holds if and only if $M\in \mathcal{K}^2(6)$.
	
				If we eliminate $f_4$, $f_5$ and $f_6$ from the Dehn-Sommerville-equations for
		combinatorial $6$-manifolds, we obtain the linear equation
				\begin{equation}
					35f_0(M)-15f_1(M)+5f_2(M)-f_3(M)=35\chi(M).
			\label{eq:DS6Manifolds}
		\end{equation}

		Inserting inequality (\ref{eq:Ineq5manif3}) into (\ref{eq:DS6Manifolds}) and solving for $f_2(M)$ yields the claimed lower bounds (\ref{eq:Ineqf2K26}) and (\ref{eq:Ineqf2K26Neigh}),
				\begin{equation}
					\begin{array}{l@{}l@{}l}
				f_2(M)&\geq&28\chi(M) - 21 f_0(M) + 6 f_1(M)\\ 
				&=&28\chi(M)+3f_0(\underbrace{f_0(M)-8}_{\geq 0}),
			\end{array}
			\label{eq:f2stacked6mnf}
		\end{equation}
				where the $2$-neighborliness of $M$ was used in the last line.
	\end{proof}
	
	For a possible $14$-vertex triangulation of $S^4\times S^2$ (with $\chi=4$) inequality (\ref{eq:f2stacked6mnf}) becomes
		\begin{equation*}
		f_2\geq 4\cdot 28 +3\cdot 14\cdot (14-8)=364,
	\end{equation*}
		but together with the trivial upper bound $f_2 \leq \binom{f_0}{3}$ this already would imply that such a triangulation necessarily is $3$-neighborly, as $\binom{14}{3}=364$.
	
	So, just by asking for a $2$-neighborly combinatorial $S^4\times S^2$ on $14$ vertices that lies in $\mathcal{K}^2(6)$ already implies that
	this triangulation is $3$-neighborly. 	Also, the example would attain equality in the Brehm-Kühnel bound \cite{Brehm87CombMnfFewVert} as an example of a $1$-connected $6$-manifold with $14$ vertices. We strongly conjecture that this triangulation also would be tight, see Question~\ref{ques:kneightight}.

\section{Subcomplexes of the cross polytope}
\label{sec:SubcomplCross}

The $d$-dimensional cross polytope (or $d$-octahedron) $\beta^d$ is defined as the convex hull of the $2d$ points 

\begin{equation*}
	(0,\ldots,0,\pm1,0,\ldots,0) \in \mathbb{R}^d.
\end{equation*}

It is a simplicial and regular polytope and it is centrally-symmetric with $d$ missing edges called \emph{diagonals}, each between two antipodal points of type $(0,\ldots,0,1,0,\ldots,0)$ and $(0,\ldots,0,-1,0,\ldots,0)$. Its edge graph is the complete $d$-partite graph with two vertices in each partition, sometimes denoted by $K_2 * \cdots * K_2$. See \cite{McMullen02AbstrRegPolytopes, Gruenbaum03ConvPoly, Ziegler95LectPolytopes} for properties of regular polytopes in general. 

The boundary of the $(d+1)$-cross polytope $\beta^{d+1}$ is an obviously minimally $d$-stacked $d$-sphere as it can be obtained as the boundary of a minimally $d$-stacked $(d+1)$-ball that is given by any subdivision of $\beta^{d+1}$ along an inner diagonal.

As pointed out in Section~\ref{sec:intro}, centrally symmetric analogues of tight triangulations appear as Hamiltonian subcomplexes of cross polytopes. A \emph{centrally symmetric triangulation} is a triangulation such that there exists a combinatorial involution operating on the face lattice of the triangulation without fixed points. Any centrally symmetric triangulation thus has an even number of vertices and can be interpreted as a subcomplex of some higher dimensional cross polytope. The tightness of a centrally symmetric $(k-1)$-connected $2k$-manifold $M$ as a subcomplex of $\beta^d$ then is equivalent to $M$ being a $k$-Hamiltonian subcomplex of $\beta^d$, i.e. that $M$ is nearly $(k+1)$-neighborly, see \cite[Ch.~4]{Kuehnel95TightPolySubm}.

As it turns out, all of the known centrally symmetric triangulations of $d$-manifolds that are $k$-Hamiltonian subcomplexes of a higher dimensional cross polytope $\beta^N$ and admit a tight embedding into $\beta^N$ are members of the class $\mathcal{K}^k(d)$. This will be discussed in the following paragraphs.

\begin{corollary}
	The $16$-vertex triangulation of $(S^2\times S^2)^{\#7}_{16}$ presented in \cite{Effenberger08HamSubRegPoly} is contained in $\mathcal{K}^2(4)$ and admits a tight embedding into $\beta^8$ as shown in \cite{Effenberger08HamSubRegPoly}. 	\label{corr:s2s27}
\end{corollary}

\begin{proof}
	The triangulation $(S^2\times S^2)^{\#7}_{16}$ is a combinatorial manifold and a tight subcomplex of $\beta^8$ as shown in \cite{Effenberger08HamSubRegPoly}. Thus, each vertex link is a PL $3$-sphere. It remains to show that all vertex links are $2$-stacked.
	
	Using \texttt{simpcomp}, we found that the vertex links can be obtained from the boundary of a $4$-simplex by a sequence of $0$- and $1$-moves. Therefore, by Lemma~\ref{lem:kstackbistellar}, the vertex links are $2$-stacked $3$-spheres. Thus, $(S^2\times S^2)^{\#7}_{16}\in \mathcal{K}^2(4)$, as claimed. 	
\end{proof}

The following centrally symmetric triangulation of $S^4\times S^2$ is a new example of a triangulation that can be seen as a subcomplex of a higher dimensional cross polytope. 
\begin{theorem}
	There exists an example of a centrally symmetric triangulation $M^6_{16}$ of $S^4\times S^2$ with $16$ vertices that is a $2$-Hamiltonian subcomplex of the $8$-octahedron $\beta^8$ and that is member of $\mathcal{K}^2(6)$. 	\label{thm:ExampleS2S4Beta8}
\end{theorem}

\begin{proof}
	The construction of $M^6_{16}$ was done entirely with \texttt{simpcomp} and is as follows. First, a $24$-vertex triangulation $\tilde{M}^6$ of $S^4\times S^2$ was constructed as the standard simplicial cartesian product of $\partial \Delta^3$ and $\partial\Delta^5$ as implemented in \cite{simpcomp}, where $\Delta^d$ denotes the $d$-simplex. Then $\tilde{M}^6$ obviously is a combinatorial $6$-manifold homeomorphic to $S^4\times S^2$.
	
	This triangulation $\tilde{M}^6$ was then reduced to the triangulation $M^6_{16}$ with $f$-vector  $f=(16, 112, 448, 980, 1232, 840, 240)$ using a vertex reduction algorithm based on bistellar flips that is implemented in \cite{simpcomp}. The code is based on the vertex reduction methods developed by Bj\"orner and Lutz \cite{Bjoerner00SimplMnfBistellarFlips}. It is well-known that this reduction process leaves the PL type of the triangulation invariant, such that $M^6_{16}\cong S^4\times S^2$ holds. 
		The $f$-vector of $M^6_{16}$ is uniquely determined already by the condition of $M^6_{16}$ to be $2$-Hamiltonian in the $8$-dimensional cross polytope.
		In particular, $M^6_{16}$ has $8$ missing edges of the form $\langle i,\,i+1\rangle$ for all odd $1\leq i \leq 15$, which are pairwise disjoint and correspond to the $8$ diagonals of the cross polytope.
		As there is an involution $$I=(1,2)(3,4)(5,6)(7,8)(9,10)(11,12)(13,14)(15,16)$$ operating on the faces of $M^6_{16}$ without fixed points, $M^6_{16}$ can be seen as a $2$-Hamiltonian subcomplex of $\beta^8$. Apart from $I$, $M^6_{16}$ has no non-trivial symmetries, i.e we have $\text{Aut}(M^{6}_{16})=\langle I\rangle \cong C_2$. The $240$ facets of $M^6_{16}$ are given in Table~\ref{tab:S2S4CentSymmSimplices}.  
	
		It remains to show that $M^6_{16}\in \mathcal{K}^2(6)$. 		Remember, that the necessary and sufficient condition for a triangulation $X$ to be member of $\mathcal{K}^k(d)$ is that all vertex links of $X$ are $k$-stacked $(d-1)$-spheres.
		Since $M^6_{16}$ is a combinatorial $6$-manifold, all vertex links are PL $5$-spheres. It thus suffices to show that all vertex links are $2$-stacked.
		
		Using \texttt{simpcomp}, we found that the vertex links can be obtained from the boundary of the $6$-simplex by a sequence of $0$- and $1$-moves. Therefore, by Lemma~\ref{lem:kstackbistellar}, vertex links are $2$-stacked $5$-spheres. Thus, $M^6_{16}\in \mathcal{K}^2(6)$, as claimed.
		
	\end{proof}

The triangulation $M^6_{16}$ is strongly conjectured to be tight in $\beta^8$. It is part of a conjectured series of centrally symmetric triangulations of sphere products for which tight embeddings into cross polytopes are conjectured. In some cases, the tightness of the embedding could be proved, see \cite{Sparla99LBTComb2kMnf}, \cite[6.2]{Kuehnel99CensusTight} and \cite[Sect.~6]{Effenberger08HamSubRegPoly}.
In particular, the sphere products presented in \cite[Thm.~6.3]{Kuehnel99CensusTight} are part of this conjectured series and the following holds.

\begin{theorem}
	The centrally symmetric triangulations of sphere products of the form $S^k\times S^m$ with vertex transitive automorphism groups
		\begin{equation*}
		\begin{array}{lllllll}
			S^1\times S^1, & S^2\times S^1, & S^3\times S^1, & S^4\times S^1, & S^5\times S^1, & S^6\times S^1, & S^7\times S^1, \\ 
			& & S^2\times S^2, & S^3\times S^2, & & S^5\times S^2, & \\ 
			& & & & S^3\times S^3, & S^4\times S^3, & S^5\times S^3, \\ 
			& & & & & & S^4\times S^4
		\end{array}
	\end{equation*}
		on $n=2(k+m)+4$ vertices presented in \cite[Thm.~6.3]{Kuehnel99CensusTight} all lie in the class $\mathcal{K}^{\min\set{k,m}}(k+m)$.
	\label{thm:centrsymmseries}
\end{theorem}

Using \texttt{simpcomp}, we found that the vertex links of all the manifolds mentioned in the statement can be obtained from the boundary of a $(k+m)$-simplex by sequences of bistellar $i$-moves, $0\leq i\leq \text{min}\{k,l\}-1$. Therefore, by Lemma~\ref{lem:kstackbistellar}, the vertex links are $\min\set{k,m}$-stacked $(k+m-1)$-spheres. Thus, all the manifolds mentioned in the statement are in $\mathcal{K}^{\min\set{k,m}}(k+m)$. Note that since these examples all have a transitive automorphism group, it suffices to check the stackedness condition for one vertex link only.

The preceding observations naturally lead to the following Question~\ref{ques:kHamilKkdTight} as a generalization of Question~\ref{ques:kneightight}.

\begin{question}
	Let $d\geq 4$ and let $M$ be a $k$-Hamiltonian codimension $2$ subcomplex of the $(d+2)$-dimensional cross polytope $\beta^{d+2}$ such that $M\in \mathcal{K}^k(d)$ for some fixed $1\leq k\leq \lceil \frac{d-1}{2} \rceil$. Does this imply that the embedding $M\subset \beta^{d+2}\subset E^{d+2}$ is tight?
 	\label{ques:kHamilKkdTight}
\end{question}

This is true for all currently known codimension $2$ subcomplexes of cross polytopes that fulfill the prerequisites of Question~\ref{ques:kHamilKkdTight}: The $8$-vertex triangulation of the torus, a $12$-vertex triangulation of $S^2\times S^2$ due to Sparla \cite{Lassmann00ClassCentSymmCycS2S2, Sparla97GeomKombEigTrigMgf}
and the triangulations of $S^k\times S^k$ on $4k+4$ vertices for $k=3$ and $k=4$ as well as for the infinite series of triangulations of $S^k\times S^1$ in \cite{Kuehnel86HigherDimCsaszar}. For the other triangulations of $S^k\times S^m$ listed in Theorem~\ref{thm:centrsymmseries} above, Kühnel and Lutz ``strongly conjecture'' \cite[Sec.~6]{Kuehnel99CensusTight} that they are tight in the $(k+m+2)$-dimensional cross polytope. Nevertheless, it is currently not clear whether the conditions of Question~\ref{ques:kHamilKkdTight} imply the tightness of the embedding into the cross polytope.

In accordance with \cite[Conjecture 6.2]{Kuehnel99CensusTight} we then have the following conjecture.

\begin{conjecture}
	Any centrally symmetric combinatorial triangulation $M^{k+m}_n$ of $S^k\times S^m$ on $n=2(k+m+2)$ vertices is tight if regarded as a subcomplex of the $\frac{n}{2}$-dimensional cross polytope. $M^{k+m}_n$ is contained in the class $\mathcal{K}^{\min\set{k,m}}(k+m)$.
\end{conjecture}

\section*{Acknowledgment}	
The author acknowledges support by the Deutsche Forschungsgemeinschaft (DFG). This work was carried out as part of the DFG project Ku 1203/5-2. 

\begin{table}
\caption{The 240 $5$-simplices of $M^6_{16}$.}
\label{tab:S2S4CentSymmSimplices}
\centering
{\tiny
\begin{tabular}{lllll}
$\langle 1\,2\,3\,4\,7\,12\,14 \rangle$, &
$\langle 1\,2\,3\,4\,7\,12\,16 \rangle$, &
$\langle 1\,2\,3\,4\,7\,13\,14 \rangle$, &
$\langle 1\,2\,3\,4\,7\,13\,16 \rangle$, &
$\langle 1\,2\,3\,4\,9\,12\,14 \rangle$,\\ 

$\langle 1\,2\,3\,4\,9\,12\,16 \rangle$, &
$\langle 1\,2\,3\,4\,9\,14\,16 \rangle$, &
$\langle 1\,2\,3\,4\,13\,14\,16 \rangle$, &
$\langle 1\,2\,3\,6\,7\,12\,14 \rangle$, &
$\langle 1\,2\,3\,6\,7\,12\,16 \rangle$,\\ 

$\langle 1\,2\,3\,6\,7\,13\,14 \rangle$, &
$\langle 1\,2\,3\,6\,7\,13\,16 \rangle$, &
$\langle 1\,2\,3\,6\,9\,10\,12 \rangle$, &
$\langle 1\,2\,3\,6\,9\,10\,13 \rangle$, &
$\langle 1\,2\,3\,6\,9\,12\,16 \rangle$,\\ 

$\langle 1\,2\,3\,6\,9\,13\,16 \rangle$, &
$\langle 1\,2\,3\,6\,10\,11\,12 \rangle$, &
$\langle 1\,2\,3\,6\,10\,11\,13 \rangle$, &
$\langle 1\,2\,3\,6\,11\,12\,14 \rangle$, &
$\langle 1\,2\,3\,6\,11\,13\,14 \rangle$,\\ 

$\langle 1\,2\,3\,9\,10\,11\,12 \rangle$, &
$\langle 1\,2\,3\,9\,10\,11\,13 \rangle$, &
$\langle 1\,2\,3\,9\,11\,12\,14 \rangle$, &
$\langle 1\,2\,3\,9\,11\,13\,14 \rangle$, &
$\langle 1\,2\,3\,9\,13\,14\,16 \rangle$,\\ 

$\langle 1\,2\,4\,7\,12\,14\,15 \rangle$, &
$\langle 1\,2\,4\,7\,12\,15\,16 \rangle$, &
$\langle 1\,2\,4\,7\,13\,14\,15 \rangle$, &
$\langle 1\,2\,4\,7\,13\,15\,16 \rangle$, &
$\langle 1\,2\,4\,9\,12\,14\,16 \rangle$,\\ 

$\langle 1\,2\,4\,12\,14\,15\,16 \rangle$, &
$\langle 1\,2\,4\,13\,14\,15\,16 \rangle$, &
$\langle 1\,2\,6\,7\,12\,14\,16 \rangle$, &
$\langle 1\,2\,6\,7\,13\,14\,15 \rangle$, &
$\langle 1\,2\,6\,7\,13\,15\,16 \rangle$,\\ 

$\langle 1\,2\,6\,7\,14\,15\,16 \rangle$, &
$\langle 1\,2\,6\,9\,10\,11\,12 \rangle$, &
$\langle 1\,2\,6\,9\,10\,11\,13 \rangle$, &
$\langle 1\,2\,6\,9\,11\,12\,14 \rangle$, &
$\langle 1\,2\,6\,9\,11\,13\,15 \rangle$,\\ 

$\langle 1\,2\,6\,9\,11\,14\,15 \rangle$, &
$\langle 1\,2\,6\,9\,12\,14\,16 \rangle$, &
$\langle 1\,2\,6\,9\,13\,15\,16 \rangle$, &
$\langle 1\,2\,6\,9\,14\,15\,16 \rangle$, &
$\langle 1\,2\,6\,11\,13\,14\,15 \rangle$,\\ 

$\langle 1\,2\,7\,12\,14\,15\,16 \rangle$, &
$\langle 1\,2\,9\,11\,13\,14\,15 \rangle$, &
$\langle 1\,2\,9\,13\,14\,15\,16 \rangle$, &
$\langle 1\,3\,4\,7\,12\,14\,16 \rangle$, &
$\langle 1\,3\,4\,7\,13\,14\,16 \rangle$,\\ 

$\langle 1\,3\,4\,9\,12\,14\,16 \rangle$, &
$\langle 1\,3\,6\,7\,12\,14\,16 \rangle$, &
$\langle 1\,3\,6\,7\,13\,14\,16 \rangle$, &
$\langle 1\,3\,6\,8\,9\,10\,11 \rangle$, &
$\langle 1\,3\,6\,8\,9\,10\,13 \rangle$,\\ 

$\langle 1\,3\,6\,8\,9\,11\,14 \rangle$, &
$\langle 1\,3\,6\,8\,9\,13\,14 \rangle$, &
$\langle 1\,3\,6\,8\,10\,11\,13 \rangle$, &
$\langle 1\,3\,6\,8\,11\,13\,14 \rangle$, &
$\langle 1\,3\,6\,9\,10\,11\,12 \rangle$,\\ 

$\langle 1\,3\,6\,9\,11\,12\,14 \rangle$, &
$\langle 1\,3\,6\,9\,12\,14\,16 \rangle$, &
$\langle 1\,3\,6\,9\,13\,14\,16 \rangle$, &
$\langle 1\,3\,8\,9\,10\,11\,13 \rangle$, &
$\langle 1\,3\,8\,9\,11\,13\,14 \rangle$,\\ 

$\langle 1\,4\,7\,8\,10\,11\,13 \rangle$, &
$\langle 1\,4\,7\,8\,10\,11\,15 \rangle$, &
$\langle 1\,4\,7\,8\,10\,13\,16 \rangle$, &
$\langle 1\,4\,7\,8\,10\,15\,16 \rangle$, &
$\langle 1\,4\,7\,8\,11\,13\,15 \rangle$,\\ 

$\langle 1\,4\,7\,8\,12\,14\,15 \rangle$, &
$\langle 1\,4\,7\,8\,12\,14\,16 \rangle$, &
$\langle 1\,4\,7\,8\,12\,15\,16 \rangle$, &
$\langle 1\,4\,7\,8\,13\,14\,15 \rangle$, &
$\langle 1\,4\,7\,8\,13\,14\,16 \rangle$,\\ 

$\langle 1\,4\,7\,10\,11\,13\,15 \rangle$, &
$\langle 1\,4\,7\,10\,13\,15\,16 \rangle$, &
$\langle 1\,4\,8\,10\,11\,13\,15 \rangle$, &
$\langle 1\,4\,8\,10\,13\,15\,16 \rangle$, &
$\langle 1\,4\,8\,12\,14\,15\,16 \rangle$,\\ 

$\langle 1\,4\,8\,13\,14\,15\,16 \rangle$, &
$\langle 1\,6\,7\,8\,10\,11\,13 \rangle$, &
$\langle 1\,6\,7\,8\,10\,11\,15 \rangle$, &
$\langle 1\,6\,7\,8\,10\,13\,16 \rangle$, &
$\langle 1\,6\,7\,8\,10\,15\,16 \rangle$,\\ 

$\langle 1\,6\,7\,8\,11\,13\,15 \rangle$, &
$\langle 1\,6\,7\,8\,13\,14\,15 \rangle$, &
$\langle 1\,6\,7\,8\,13\,14\,16 \rangle$, &
$\langle 1\,6\,7\,8\,14\,15\,16 \rangle$, &
$\langle 1\,6\,7\,10\,11\,13\,15 \rangle$,\\ 

$\langle 1\,6\,7\,10\,13\,15\,16 \rangle$, &
$\langle 1\,6\,8\,9\,10\,11\,15 \rangle$, &
$\langle 1\,6\,8\,9\,10\,13\,16 \rangle$, &
$\langle 1\,6\,8\,9\,10\,15\,16 \rangle$, &
$\langle 1\,6\,8\,9\,11\,14\,15 \rangle$,\\ 

$\langle 1\,6\,8\,9\,13\,14\,16 \rangle$, &
$\langle 1\,6\,8\,9\,14\,15\,16 \rangle$, &
$\langle 1\,6\,8\,11\,13\,14\,15 \rangle$, &
$\langle 1\,6\,9\,10\,11\,13\,15 \rangle$, &
$\langle 1\,6\,9\,10\,13\,15\,16 \rangle$,\\ 

$\langle 1\,7\,8\,12\,14\,15\,16 \rangle$, &
$\langle 1\,8\,9\,10\,11\,13\,15 \rangle$, &
$\langle 1\,8\,9\,10\,13\,15\,16 \rangle$, &
$\langle 1\,8\,9\,11\,13\,14\,15 \rangle$, &
$\langle 1\,8\,9\,13\,14\,15\,16 \rangle$,\\ 

$\langle 2\,3\,4\,5\,7\,10\,11 \rangle$, &
$\langle 2\,3\,4\,5\,7\,10\,16 \rangle$, &
$\langle 2\,3\,4\,5\,7\,11\,14 \rangle$, &
$\langle 2\,3\,4\,5\,7\,14\,16 \rangle$, &
$\langle 2\,3\,4\,5\,9\,10\,11 \rangle$,\\ 

$\langle 2\,3\,4\,5\,9\,10\,12 \rangle$, &
$\langle 2\,3\,4\,5\,9\,11\,14 \rangle$, &
$\langle 2\,3\,4\,5\,9\,12\,16 \rangle$, &
$\langle 2\,3\,4\,5\,9\,14\,16 \rangle$, &
$\langle 2\,3\,4\,5\,10\,12\,16 \rangle$,\\ 

$\langle 2\,3\,4\,7\,10\,11\,12 \rangle$, &
$\langle 2\,3\,4\,7\,10\,12\,16 \rangle$, &
$\langle 2\,3\,4\,7\,11\,12\,14 \rangle$, &
$\langle 2\,3\,4\,7\,13\,14\,16 \rangle$, &
$\langle 2\,3\,4\,9\,10\,11\,12 \rangle$,\\ 

$\langle 2\,3\,4\,9\,11\,12\,14 \rangle$, &
$\langle 2\,3\,5\,6\,9\,10\,12 \rangle$, &
$\langle 2\,3\,5\,6\,9\,10\,13 \rangle$, &
$\langle 2\,3\,5\,6\,9\,11\,13 \rangle$, &
$\langle 2\,3\,5\,6\,9\,11\,14 \rangle$,\\ 

$\langle 2\,3\,5\,6\,9\,12\,16 \rangle$, &
$\langle 2\,3\,5\,6\,9\,14\,16 \rangle$, &
$\langle 2\,3\,5\,6\,10\,11\,12 \rangle$, &
$\langle 2\,3\,5\,6\,10\,11\,13 \rangle$, &
$\langle 2\,3\,5\,6\,11\,12\,14 \rangle$,\\ 

$\langle 2\,3\,5\,6\,12\,14\,16 \rangle$, &
$\langle 2\,3\,5\,7\,10\,11\,12 \rangle$, &
$\langle 2\,3\,5\,7\,10\,12\,16 \rangle$, &
$\langle 2\,3\,5\,7\,11\,12\,14 \rangle$, &
$\langle 2\,3\,5\,7\,12\,14\,16 \rangle$,\\ 

$\langle 2\,3\,5\,9\,10\,11\,13 \rangle$, &
$\langle 2\,3\,6\,7\,12\,14\,16 \rangle$, &
$\langle 2\,3\,6\,7\,13\,14\,16 \rangle$, &
$\langle 2\,3\,6\,9\,11\,13\,14 \rangle$, &
$\langle 2\,3\,6\,9\,13\,14\,16 \rangle$,\\ 

$\langle 2\,4\,5\,7\,10\,11\,12 \rangle$, &
$\langle 2\,4\,5\,7\,10\,12\,15 \rangle$, &
$\langle 2\,4\,5\,7\,10\,15\,16 \rangle$, &
$\langle 2\,4\,5\,7\,11\,12\,14 \rangle$, &
$\langle 2\,4\,5\,7\,12\,14\,15 \rangle$,\\ 

$\langle 2\,4\,5\,7\,14\,15\,16 \rangle$, &
$\langle 2\,4\,5\,9\,10\,11\,12 \rangle$, &
$\langle 2\,4\,5\,9\,11\,12\,14 \rangle$, &
$\langle 2\,4\,5\,9\,12\,14\,16 \rangle$, &
$\langle 2\,4\,5\,10\,12\,15\,16 \rangle$,\\ 

$\langle 2\,4\,5\,12\,14\,15\,16 \rangle$, &
$\langle 2\,4\,7\,10\,12\,15\,16 \rangle$, &
$\langle 2\,4\,7\,13\,14\,15\,16 \rangle$, &
$\langle 2\,5\,6\,9\,10\,11\,12 \rangle$, &
$\langle 2\,5\,6\,9\,10\,11\,13 \rangle$,\\ 

$\langle 2\,5\,6\,9\,11\,12\,14 \rangle$, &
$\langle 2\,5\,6\,9\,12\,14\,16 \rangle$, &
$\langle 2\,5\,7\,10\,12\,15\,16 \rangle$, &
$\langle 2\,5\,7\,12\,14\,15\,16 \rangle$, &
$\langle 2\,6\,7\,13\,14\,15\,16 \rangle$,\\ 

$\langle 2\,6\,9\,11\,13\,14\,15 \rangle$, &
$\langle 2\,6\,9\,13\,14\,15\,16 \rangle$, &
$\langle 3\,4\,5\,7\,8\,10\,11 \rangle$, &
$\langle 3\,4\,5\,7\,8\,10\,16 \rangle$, &
$\langle 3\,4\,5\,7\,8\,11\,12 \rangle$,\\ 

$\langle 3\,4\,5\,7\,8\,12\,16 \rangle$, &
$\langle 3\,4\,5\,7\,11\,12\,14 \rangle$, &
$\langle 3\,4\,5\,7\,12\,14\,16 \rangle$, &
$\langle 3\,4\,5\,8\,9\,10\,11 \rangle$, &
$\langle 3\,4\,5\,8\,9\,10\,12 \rangle$,\\ 

$\langle 3\,4\,5\,8\,9\,11\,12 \rangle$, &
$\langle 3\,4\,5\,8\,10\,12\,16 \rangle$, &
$\langle 3\,4\,5\,9\,11\,12\,14 \rangle$, &
$\langle 3\,4\,5\,9\,12\,14\,16 \rangle$, &
$\langle 3\,4\,7\,8\,10\,11\,12 \rangle$,\\ 

$\langle 3\,4\,7\,8\,10\,12\,16 \rangle$, &
$\langle 3\,4\,8\,9\,10\,11\,12 \rangle$, &
$\langle 3\,5\,6\,8\,9\,10\,12 \rangle$, &
$\langle 3\,5\,6\,8\,9\,10\,13 \rangle$, &
$\langle 3\,5\,6\,8\,9\,11\,12 \rangle$,\\ 

$\langle 3\,5\,6\,8\,9\,11\,13 \rangle$, &
$\langle 3\,5\,6\,8\,10\,11\,12 \rangle$, &
$\langle 3\,5\,6\,8\,10\,11\,13 \rangle$, &
$\langle 3\,5\,6\,9\,11\,12\,14 \rangle$, &
$\langle 3\,5\,6\,9\,12\,14\,16 \rangle$,\\ 

$\langle 3\,5\,7\,8\,10\,11\,12 \rangle$, &
$\langle 3\,5\,7\,8\,10\,12\,16 \rangle$, &
$\langle 3\,5\,8\,9\,10\,11\,13 \rangle$, &
$\langle 3\,6\,8\,9\,10\,11\,12 \rangle$, &
$\langle 3\,6\,8\,9\,11\,13\,14 \rangle$,\\ 

$\langle 4\,5\,7\,8\,10\,11\,13 \rangle$, &
$\langle 4\,5\,7\,8\,10\,13\,16 \rangle$, &
$\langle 4\,5\,7\,8\,11\,12\,15 \rangle$, &
$\langle 4\,5\,7\,8\,11\,13\,15 \rangle$, &
$\langle 4\,5\,7\,8\,12\,14\,15 \rangle$,\\ 

$\langle 4\,5\,7\,8\,12\,14\,16 \rangle$, &
$\langle 4\,5\,7\,8\,13\,14\,15 \rangle$, &
$\langle 4\,5\,7\,8\,13\,14\,16 \rangle$, &
$\langle 4\,5\,7\,10\,11\,12\,15 \rangle$, &
$\langle 4\,5\,7\,10\,11\,13\,15 \rangle$,\\ 

$\langle 4\,5\,7\,10\,13\,15\,16 \rangle$, &
$\langle 4\,5\,7\,13\,14\,15\,16 \rangle$, &
$\langle 4\,5\,8\,9\,10\,11\,13 \rangle$, &
$\langle 4\,5\,8\,9\,10\,12\,15 \rangle$, &
$\langle 4\,5\,8\,9\,10\,13\,15 \rangle$,\\ 

$\langle 4\,5\,8\,9\,11\,12\,15 \rangle$, &
$\langle 4\,5\,8\,9\,11\,13\,15 \rangle$, &
$\langle 4\,5\,8\,10\,12\,15\,16 \rangle$, &
$\langle 4\,5\,8\,10\,13\,15\,16 \rangle$, &
$\langle 4\,5\,8\,12\,14\,15\,16 \rangle$,\\ 

$\langle 4\,5\,8\,13\,14\,15\,16 \rangle$, &
$\langle 4\,5\,9\,10\,11\,12\,15 \rangle$, &
$\langle 4\,5\,9\,10\,11\,13\,15 \rangle$, &
$\langle 4\,7\,8\,10\,11\,12\,15 \rangle$, &
$\langle 4\,7\,8\,10\,12\,15\,16 \rangle$,\\ 

$\langle 4\,8\,9\,10\,11\,12\,15 \rangle$, &
$\langle 4\,8\,9\,10\,11\,13\,15 \rangle$, &
$\langle 5\,6\,7\,8\,10\,11\,13 \rangle$, &
$\langle 5\,6\,7\,8\,10\,11\,15 \rangle$, &
$\langle 5\,6\,7\,8\,10\,13\,15 \rangle$,\\ 

$\langle 5\,6\,7\,8\,11\,13\,15 \rangle$, &
$\langle 5\,6\,7\,10\,11\,13\,15 \rangle$, &
$\langle 5\,6\,8\,9\,10\,12\,15 \rangle$, &
$\langle 5\,6\,8\,9\,10\,13\,15 \rangle$, &
$\langle 5\,6\,8\,9\,11\,12\,15 \rangle$,\\ 

$\langle 5\,6\,8\,9\,11\,13\,15 \rangle$, &
$\langle 5\,6\,8\,10\,11\,12\,15 \rangle$, &
$\langle 5\,6\,9\,10\,11\,12\,15 \rangle$, &
$\langle 5\,6\,9\,10\,11\,13\,15 \rangle$, &
$\langle 5\,7\,8\,10\,11\,12\,15 \rangle$,\\ 

$\langle 5\,7\,8\,10\,12\,15\,16 \rangle$, &
$\langle 5\,7\,8\,10\,13\,15\,16 \rangle$, &
$\langle 5\,7\,8\,12\,14\,15\,16 \rangle$, &
$\langle 5\,7\,8\,13\,14\,15\,16 \rangle$, &
$\langle 6\,7\,8\,10\,13\,15\,16 \rangle$,\\ 

$\langle 6\,7\,8\,13\,14\,15\,16 \rangle$, &
$\langle 6\,8\,9\,10\,11\,12\,15 \rangle$, &
$\langle 6\,8\,9\,10\,13\,15\,16 \rangle$, &
$\langle 6\,8\,9\,11\,13\,14\,15 \rangle$, &
$\langle 6\,8\,9\,13\,14\,15\,16 \rangle$.
\end{tabular}} \end{table}

\bibliographystyle{plain}
\footnotesize
\bibliography{/home/effenbfx/bib/bibliography}
\end{document}